\documentclass[12pt, final]{report}   
\usepackage{graphicx}  
\usepackage[left=1.2in, top=1in, right=1.2in, bottom=1in, foot=0.2in, head=-0.05in]{geometry}
\usepackage{setspace}  
\usepackage{fancyhdr}
\usepackage[explicit]{titlesec}  
\usepackage[titles]{tocloft}  
\usepackage[normalem]{ulem}
\usepackage[backend=bibtex, sorting=none, bibstyle=ieee]{biblatex}  
\usepackage{appendix}  
\usepackage{rotating}  
\usepackage{ulem}  
\usepackage{textcomp} 
\usepackage{indentfirst} 
\usepackage{booktabs,array,arydshln} 
\usepackage{amsmath} 
\usepackage{amsfonts}
\usepackage[T1]{fontenc} 
\usepackage[utf8]{inputenc} 
\usepackage[colorlinks=true, linkcolor=black, urlcolor=black, citecolor=black, anchorcolor=blue]{hyperref}

\usepackage{color}
\usepackage{mathbbol}
\usepackage[mathscr, mathcal]{eucal}
\usepackage{amscd}
\usepackage{mathrsfs,pifont}
\usepackage{amssymb,amsfonts}
\usepackage[all]{xy}
\usepackage{wrapfig}
\usepackage{subcaption}

\usepackage{setspace}

\usepackage{tikz}
\newenvironment{proof}{\paragraph{Proof:}}{\hfill$\square$}

\bibliography{references}

\AtEveryBibitem{\clearfield{issn}}
\AtEveryBibitem{\clearlist{issn}}

\AtEveryBibitem{\clearfield{language}}
\AtEveryBibitem{\clearlist{language}}

\AtEveryBibitem{\clearfield{doi}}
\AtEveryBibitem{\clearlist{doi}}

\AtEveryBibitem{\clearfield{url}}
\AtEveryBibitem{\clearlist{url}}

\AtEveryBibitem{%
  \ifentrytype{online}
    {}
    {\clearfield{urlyear}\clearfield{urlmonth}\clearfield{urlday}}}

\def\R{\mathbb R}
\def\C{\mathbb C}
\def\Q{\mathbb Q}
\def\Z{\mathbb Z}
\def\L{\mathfrak{L}}

\def\P{\mathbb P}

\def\qq{\mathbb q}

\def\1{\mathbbm 1}

\def\g{\mathfrak{g}}

\def\fC{\mathfrak{C}}
\def\fD{\mathfrak{D}}

\def\m{\mathfrak{m}}
\def\n{\mathfrak{n}}
\def\L{\mathscr{L}}

\def\V{\mathscr{V}}

\def\cT{\mathscr{T}}
\def\cO{\mathscr{O}}
\def\cK{\mathscr{K}}
\def\cU{\mathscr{U}}
\def\cA{\mathscr{A}}

\def\B{\mathsf{B}}

\def\Sh{\mathsf{Sh}}
\def\Rt{\mathsf{Rt}}

\def\supp{\text{supp}\,}

\newcommand{\inner}[2]{\langle #1,#2\, \rangle}

\newcommand{\interior}[1]{\raise0.2ex\hbox{$\displaystyle{\mathop{#1}^{\circ}}$}}

\newcommand{\intp}[1]{\lfloor #1 \rfloor}

\renewcommand\phi{\varphi}

\newtheorem{theorem}{Theorem}[section]

\newtheorem{proposition}[theorem]{Proposition}

\newtheorem{corollary}[theorem]{Corollary}
\newtheorem{remark}[theorem]{Remark}
\newtheorem{conjecture}[theorem]{Conjecture}
\newtheorem{lemma}[theorem]{Lemma}
\newtheorem{example}[theorem]{Example}
\numberwithin{equation}{section}

\newcommand{\id}{\mathrm{id}}

\newcommand{\tr}{\mathrm{tr}\,}

\newcommand{\Hom}{\mathrm{Hom}}

\newcommand{\End}{\mathrm{End}}

\newcommand{\Ell}{\mathrm{Ell}}

\newcommand{\Lie}{\mathrm{Lie}\,}

\newcommand{\NS}{\mathrm{NS}}
\newcommand{\diag}{\mathrm{diag}}
\newcommand{\Attr}{\mathrm{Attr}}
\newcommand{\Spec}{\mathrm{Spec}\,}

\newcommand{\Hilb}{\mathrm{Hilb}}

\newcommand{\Sym}{\mathrm{Sym}}

\newcommand{\vir}{\mathrm{vir}}

\newcommand{\Pic}{\mathrm{Pic}}

\newcommand{\eps}{\varepsilon}

\newcommand{\Stabb}{\mathrm{\bf Stab}}
\newcommand{\Stab}{\mathrm{Stab}}
\newcommand{\stab}{\mathrm{stab}}
\newcommand{\rk}{\mathrm{rk}\,}
\newcommand{\ahat}{\widehat{\mathsf{a}}}

\newcommand{\Mon}{\mathrm{Mon}}

\newcommand{\ind}{\mathrm{ind}}

\newcommand{\Glue}{\mathsf{Glue}}

\newcommand{\sA}{\mathsf{A}}
\newcommand{\sE}{\mathsf{E}}
\newcommand{\sR}{\mathsf{R}}
\newcommand{\sT}{\mathsf{T}}
\newcommand{\sK}{\mathsf{K}}
\newcommand{\sX}{\mathsf{X}}
\newcommand{\sw}{\mathsf{w}}
\newcommand{\sv}{\mathsf{v}}
\newcommand{\sH}{\mathsf{H}}
\newcommand{\Walls}{\mathsf{Walls}}
\newcommand{\Res}{\mathsf{Res}}

\newcommand{\flop}{\mathrm{flop}}

\newcommand{\QM}{\mathsf{QM}}

\begin{document}
\doublespacing  

\begin{titlepage}
\begin{center}

\begin{singlespacing}
\vspace*{6\baselineskip}
Elliptic stable envelopes and $3d$ mirror symmetry\\
\vspace{3\baselineskip}
Iakov Kononov\\
\vspace{18\baselineskip}
Submitted in partial fulfillment of the\\
requirements for the degree of\\
Doctor of Philosophy\\
under the Executive Committee\\
of the Graduate School of Arts and Sciences\\
\vspace{3\baselineskip}
COLUMBIA UNIVERSITY\\
\vspace{3\baselineskip}
\the\year
\vfill

\end{singlespacing}

\end{center}
\end{titlepage}


\begin{titlepage}
\begin{singlespacing}
\begin{center}

\vspace*{35\baselineskip}

\textcopyright  \,  \the\year\\
\vspace{\baselineskip}	
Yakov Kononov\\
\vspace{\baselineskip}	
All Rights Reserved
\end{center}
\vfill

\end{singlespacing}
\end{titlepage}

\pagenumbering{gobble}


\begin{titlepage}
\begin{center}

\vspace*{5\baselineskip}
\textbf{\large Abstract}

Elliptic stable envelopes and $3d$ mirror symmetry

Iakov Kononov
\end{center}
\hspace{10mm}
In this thesis we discuss various classical problems in enumerative geometry. We are focused on ideas and methods which can be used explicitly for practical computations. Our approach is based on studying the limits of elliptic stable envelopes with shifted equivariant or K\"ahler variables from elliptic cohomology to K-theory.

We prove that for a variety $\sX$  we can obtain K-theoretic stable envelopes for the variety $\sX^G$ of the $G$-fixed points of $\sX$, where $G$ is a cyclic group acting on $\sX$ preserving the symplectic form. 

We formalize the notion of symplectic duality, also known as 3-dimensional mirror symmetry. We obtain a factorization theorem about the limit of elliptic stable envelopes to a wall, which generalizes the result \cite{AOElliptic}. This approach allows us to extend the action of quantum groups, quantum Weyl groups \cite{EV}, R-matrices etc., to  actions on the K-theory of the symplectic dual variety. In the case of $\sX = \Hilb(\C^2, n)$, our results imply  the conjectures of E.Gorsky and A.Negut from \cite{NegGor}.

We propose a new approach to K-theoretic quantum difference equations.

\vspace*{\fill}
\end{titlepage}

\pagenumbering{roman}
\setcounter{tocdepth}{1}
\renewcommand{\cftchapdotsep}{\cftdotsep}  
\renewcommand{\cftchapfont}{\normalfont}  
\renewcommand{\cftchappagefont}{}  
\renewcommand{\cftchappresnum}{Chapter }
\renewcommand{\cftchapaftersnum}{:}
\renewcommand{\cftchapnumwidth}{5em}
\renewcommand{\cftchapafterpnum}{\vskip\baselineskip} 
\renewcommand{\cftsecafterpnum}{\vskip\baselineskip}  
\renewcommand{\cftsubsecafterpnum}{\vskip\baselineskip} 
\renewcommand{\cftsubsubsecafterpnum}{\vskip\baselineskip} 

\titleformat{\chapter}[display]
{\normalfont\bfseries\filcenter}{\chaptertitlename\ \thechapter}{0pt}{\large{#1}}

\renewcommand\contentsname{Table of Contents}

\begin{singlespace}
\tableofcontents
\end{singlespace}

\clearpage

\phantomsection


\clearpage
\begin{center}

\vspace*{1\baselineskip}
\textbf{\large Acknowledgements}
\end{center}

\hspace{7mm}

First of all, I want to thank my advisor Andrei Okounkov for his encouragement. He introduced me to a very modern and interesting mathematics. He suggested problems and my work on it was guided by his incredible intuition. Andrei has been a great advisor who has always been  happy to meet and explain his ideas with enthusiasm. The ideas he taught me are priceless and have changed my view of mathematics substantially.

I would also like to thank my friend and coathor Andrey Smirnov with whomst I spent many hours working on these problems. I enjoyed our collaboration and it was very productive. We have had  many fascinating discussions and I hope to work with him in future again. Most of the results of this dissertation is our joint work and his contribution is greatly appreciated.

I thank Alexander Belavin,  Doron Gepner, Misha Khovanov, Alexey Morozov, Nikita Nekrasov, Victor Ostrik, Boris Feigin  for their unwavering support and encouragement.

I benefitted a lot from numerous discussions with  Mina Aganagic, Konstantin Aleshkin, Noah Arbesfeld, Ivan Danilenko, Pavel Etingof, Igor Krichever, Henry Liu, Melissa Liu, Michael McBreen, Anton Osinenko, Peter Pushkar, Gus Schrader, Shamil Shakirov, Nikita Sopenko, Shuai Wang, Yegor Zenkevich.

I would like to thank all my friends at Columbia and New York. Fellow graduate students, old friends and new friends, who made this such a fun and exciting time for me. The same can be said about my old friends in Moscow.

It gives me great pleasure to thank my family for their joy, endless support and never letting me alone.
 
\setcounter{page}{4}

\addcontentsline{toc}{chapter}{Acknowledgments}

\clearpage


\phantomsection
\addcontentsline{toc}{chapter}{Dedication}

\begin{center}

\vspace*{5\baselineskip}
\textbf{\large Dedication}
\end{center}

\begin{flushleft}
\hspace{10mm}
This dissertation is lovingly dedicated to my mother for her love, support and encouragement.
\end{flushleft}






\newpage

\clearpage
\pagenumbering{arabic}
\setcounter{page}{1} 

\phantomsection
\addcontentsline{toc}{chapter}{Introduction}

\begin{center}
\vspace*{5\baselineskip}
\textbf{\large Introduction}
\end{center}


\titleformat{\chapter}[display]
{\normalfont\bfseries\filcenter}{}{0pt}{\large\chaptertitlename\ \large\thechapter : \large\bfseries\filcenter{#1}}  
\titlespacing*{\chapter}
  {0pt}{0pt}{30pt}	
  
\titleformat{\section}{\normalfont\bfseries}{\thesection}{1em}{#1}

\titleformat{\subsection}{\normalfont}{\thesubsection}{0em}{\hspace{1em}#1}



\section{Introduction}

Geometric representation theory constructs representations of various quantum groups using geometry.
For a Nakajima quiver variety  $\sX$ there are three levels of cohomology theories, and in each of them  certain quantum groups act.
\begin{enumerate}
\item Rational level: Equivariant cohomology $\sH_\sT(\sX)$ has an action of a Yangian $Y(\g)$.
This level of theory has been brilliantly developed in \cite{MO}, and explicitly for the instanton moduli space in \cite{InstR}.
\item Trigonometric level: Equivariant K-theory $\sK_\sT(\sX)$ has an action of a quantum group $\cU_\hbar(\g)$, see \cite{pcmilect}.
\item Elliptic level: Equivariant elliptic cohomology $\sE_\sT(\sX)$ has an action of an elliptic quantum group.
\end{enumerate}

The action of the quantum group in each case can be  reconstructed from  the $R$-matrix by taking various matrix elements with respect to auxiliary spaces. The $R$-matrix plays a crucial role in the inverse scattering method and the algebraic Bethe ansatz discovered in \cite{FRT}, and developed in \cite{OkBethe}.
From geometric point of view, the $R$-matrix is the transition matrix between different {\it stable} bases for the generalized cohomology of $\sX$.

Stable envelopes turn out to be useful in many problems. In enumerative geometry one naturally considers different notions of curve counting such as Gromov-Witten-theory, Donaldson-Thomas theory and others, see \cite{Pandharipande}. In the context of representation theory some of the most interesting moduli spaces are the moduli spaces of quasimaps. They are both rich enough to capture interesting objects and concrete enough so that in practice the computations can be done in almost combinatorial terms. 

One of the most important functions in enumerative geometry is the one-point partition function with a nonsingular boundary condition, or {\it the  equivariant vertex}. It can be defined straightforwardly as a series in the K\"ahler parameters, whose coefficients have an interesting combinatorial meaning. In the case of the Hilbert scheme $\sX = \Hilb(\C^2)$, this series is related to the equivariant Donaldson-Thomas vertex, and in general depends on  "3 legs" as an element of  $K_\sT(\Hilb(\C^2))^{\otimes 3} $. It turns out that when only 2 out of 3 legs are non-trivial, there is a beautiful closed formula for the vertex, see \cite{KOO}.

\begin{figure}[!htb]
\begin{center}
\hspace*{-2mm}
\begin{tikzpicture}%
  [scale=0.5, z=(-135:0.82cm), x=(-15:1cm), y=(90:1cm),%
  baseline={([yshift=5mm]current bounding box.south)}]
\draw (1,1,0)--(1,1,10) (1,1,0)--(10,1,0) (1,1,0)--(1,7.5,0);
\foreach \m [count=\y] in {%
      {11,11,8,6,6,5,3,3,3,3},{7,7,7,6,5,4,3,1,1,1},{7,7,6,3,3,3},{3,2,2},{2,1},{2,1},{2,1},{2,1}%
  } {
  \foreach \n [count=\x] in \m {
  \ifnum \n>0
      \foreach \z in {1,...,\n}{
        \draw[fill=black!90] (\x+1,\y,\z)--(\x+1,\y+1,\z)--(\x+1,\y+1,\z-1)--(\x+1,\y,\z-1)--cycle;
        \draw[fill=white] (\x,\y+1,\z)--(\x+1,\y+1,\z)--(\x+1,\y+1,\z-1)--(\x,\y+1,\z-1)--cycle;
        \draw[fill=gray!60] (\x,\y,\z)--(\x+1,\y,\z)--(\x+1,\y+1,\z)--(\x,\y+1,\z)--cycle;  
      }
 \fi
 }
}
\draw[fill=red] (1,1,11)--(1,2,11)--(3,2,11)--(3,1,11)--cycle;
\draw (2,2,11)--(2,1,11);
\draw[fill=red] (11,1,0)--(11,3,0)--(11,3,1)--(11,2,1)--(11,2,3)--(11,1,3)--cycle;
\draw (11,2,0)--(11,2,1)--(11,1,1);
\draw (11,2,2)--(11,1,2);
\draw[fill=red] (1,9,0)--(3,9,0)--(3,9,1)--(2,9,1)--(2,9,2)--(1,9,2)--cycle;
\draw (2,9,0)--(2,9,1)--(1,9,1);
\node[color=red] at (1,1,12) {$\lambda_1$};
\node[color=red] at (12,1,0) {$\lambda_2$};
\node[color=red] at (2,11,1) {$\lambda_3$};
\end{tikzpicture}
\caption{The vertex function $\text{Vertex}(\lambda_1, \lambda_2, \lambda_3)$ can be computed by summation of weights for 3-dimensional infinite Young diagrams with asymptotic boundary conditions $\lambda_1, \lambda_2, \lambda_3$}
\end{center}
\end{figure}
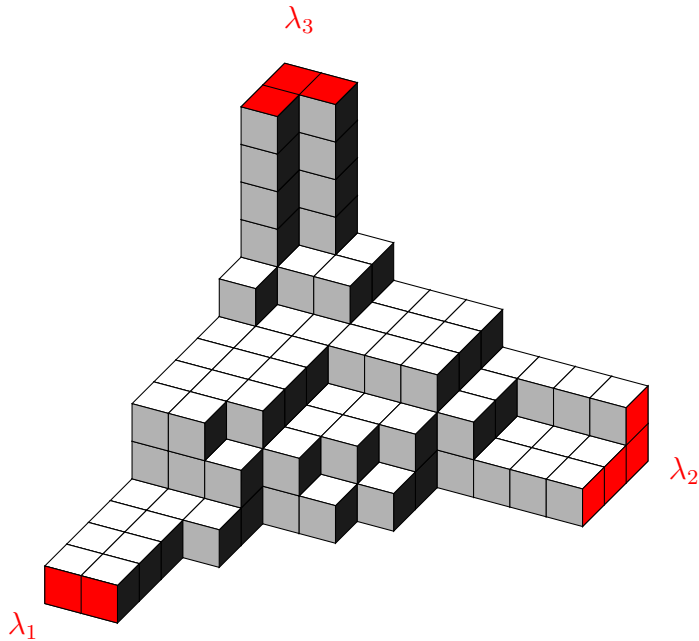

Another, more powerful, approach is to describe the equivariant vertex completely by identifying the difference equations it satisfies. There are two main types of K-theoretic difference operators: those shifting equivariant and those shifting K\"ahler variables. They both can be defined as certain two-point correlation functions. It turns out that our results about elliptic stable envelopes can be used to give  a geometric construction of quantum difference equations, as we show in joint work with Andrey Smirnov \cite{KononovSmirnov1, KononovSmirnov2}.

The results proven in this dissertation are  mathematical applications of {\it 3d mirror symmetry}, which is an important feature of $N=4$ three-dimensional supersymmetric quantum field theories. The $3d$ mirror symmetry for a $3d$ theory assigns a "mirror" theory which has the same correlation functions. One could say that the dual theories are the two "languages" to describe the same physics.

The low-energy behavior of such theories can be understood using the modulated vacua, and mathematically is governed by moduli spaces of maps to certain singular symplectic varieties.  In this approximation, $3d$ mirror symmetry relates the corresponding  moduli spaces of vacua. We expect that the enumerative and topological invariants of the $3d$ dual varieties are related in a nontrivial nonperturbative way. In chapter \ref{3dmir} we study consequences of this duality at the level of equivariant elliptic cohomology and K-theory.

In chapter \ref{qde} we give an introduction to quantum difference equations and shift operators, and outline how the results from chapter \ref{3dmir} can be used for the geometric construction of quantum difference equations. These ideas were originated by A.Okounkov in \cite{OkRodeIsland} and have continued to be developed by A.Smirnov and the author. We hope they will be published soon.

In chapter \ref{explicitexamples} we give explicit examples of the  computation of elliptic stable envelopes, their factorization, quantum difference equations and shift operators.

\section{Notation on spaces and characters}

For an algebraic torus $\sT$ over $\C$ we will use the character and cocharacter lattices
\[
cochar(\sT) = \Hom(\C^\times, \sT) \subset \Lie_\R(\sT),
\]
\[
char(\sT) = \Hom(\sT, \C^\times) \subset \Lie_\R(\sT)^\vee
\]
We denote the canonical pairing between characters and cocharacters by $\inner{-}{-}$:
\[
\Hom(\C^\times, \sT) \otimes \Hom(\sT, \C^\times) \to \Hom(\C^\times, \C^\times) \cong \Z.
\]

Let $V$ be a representation of some group $G$. The dual representation will be denoted by $V^\vee$, or, sometimes, when $G$ is a torus, as $\overline V$.
We define the symmetric algebra
\[
\mathsf{S}^\bullet V = \bigoplus_k (V^{\otimes k})^{S(k)},
\]
which is a $G$-representation with the  character
\[
\tr_{\mathsf{S}^\bullet V} = \exp \left(
\sum_{n\geq 1} \frac{1}{n} \tr_V g^n
\right).
\]
Sometimes $\mathsf{S}^\bullet$ is called the {\it plethystic exponential}.
It satisfies the relation
\[
\mathsf{S}^\bullet(V_1 \oplus V_2) = \mathsf{S}^\bullet V_1 \otimes \mathsf{S}^\bullet V_2,
\]
and thus  can be extended to K-theory. In particular,
$\mathsf{S}^\bullet (-V)$ is equal to
\[
\bigwedge\nolimits^\bullet V = \sum_{k\geq 0} (-1)^k \bigwedge\nolimits^k V
\]
We will also use a symmetrized version of the wedge power
\[
\ahat(V) = (-1)^{\dim V} (\det V)^{-1/2} \bigwedge\nolimits^\bullet V,
\]
or, explicitly, for a character
\[
\ahat\left(\sum m_i w_i\right) = \prod_i \left(\sqrt{w_i} - \frac{1}{\sqrt{w_i}}\right)^{m_i}
\]
Also we use a theta-analogue of the $\ahat$-operator, which can be written as
\[
\vartheta\left(\sum m_i w_i\right) = \prod_i \vartheta(w_i)^{m_i}
\]
for the classical odd theta-function $\vartheta(x)$. This notation should not be confused with the value of the theta function at the point $\sum m_i w_i$, since we use only multiplicative conventions on elliptic curves and theta-functions.

In the discussion of limits $q\to 0$ of elliptic cohomology classes, we will get expressions involving the floor function.
For any element $V \in K_{\sT}(pt)$
with a character
\[
V = \sum_i m_i w_i
\]
and a cocharacter $\sw \in cochar(\sT)$ define
\[
\intp{V\cdot \sw} = \sum_i m_i \intp{\inner{w_i}{\sw}}
\]

\section{Set-up}

Symplectic resolutions  can be viewed as "Lie algebras of the 21st century".
A {\it symplectic resolution} is a smooth algebraic symplectic variety $(\sX, \omega)$ such that the map
\[
\sX \to \sX_0 = \Spec H^0(\cO_\sX)
\]
is projective and birational.
It is called {\it equivariant} if there is a torus $\sT$ action on $\sX$ that scales the symplectic form $\omega$ with a character, which is denoted by $\hbar$, and contracts $\sX_0$ to a point. Given $\sX_0$, the resolution is controlled by a choice of a cohomology class
\[
\theta = [\omega] \in H^2(\sX, \R).
\]
Let $\sA = \ker \hbar$ be the subtorus preserving the symplectic form.
In enumerative geometry, we usually count curves with respect to their topological class, i.e. {\it degree} in $H^2$.
It is natural to consider the {\it K\"ahler torus} 
\[
\sK = \Pic(\sX) \otimes \C^\times = H^2(\sX, \C^\times).
\]
For brevity, we refer to coordinates in $\sA$ and $\sK$ as equivariant and K\"ahler parameters respectively.
It is natural to think of $\theta$ as an element
\[
\theta \in \Lie_\R (\sK).
\]

A polarization of $\sX$ is a choice of a virtual square root of the tangent bundle, that is a K-theoretic class $T^{1/2} \sX \in \sK_\sT(\sX)$
such that
\[
T\sX = T^{1/2} \sX + \hbar^{-1} \otimes \left(
T^{1/2} \sX
\right)^\vee.
\]
The opposite polarization is defined as
\[
T^{1/2}_{opp} \sX = T \sX - T^{1/2} \sX.
\]
If the variety  can be represented as $\sX = T^*M$ for some variety $M$, there is a natural polarization by the base $T^{1/2}\sX = TM$ directions.

For simplicity we will assume that the $\sA$-fixed points in $\sX$ are isolated.
The normal weights to fixed points are called {\it equivariant roots}. They partition $\Lie_\R(\sA)$ into finitely many {\it equivariant chambers}.

Choose a chamber $\fC \subset \Lie(\sA)$, and  pick $\sigma \in \fC$.
The attracting set of a point $p\in \sX^\sA$ is defined as
\[
\Attr_\sigma(p) = \{x\in \sX: \lim_{t\to \infty} e^{\sigma t} x = p\}.
\]
Near $p$, the set $\Attr_\sigma(p)$ is a smooth Lagrangian subvariety.

Let $\Attr^f_\sigma(p)$ be the full attracting set -- the smallest closed subset of $\sX$ containing $p$ and closed under taking $\Attr_\sigma$. It defines a partial ordering on fixed points by
\[
p_1 > p_2 \Longleftrightarrow p_2 \in \Attr_\sigma^f(p_1),
\]
which depends only on the chamber $\fC$, and does not depend on the particular choice of $\sigma$ within it.

For a fixed point $p \in \sX^\sA$ define its index to be the positive part of the polarization with respect to $\fC$:
\[
\ind_p = (N_p^{1/2})_{>0}.
\]

Each fixed point $p$ defines a map
\[
\chi_p(-, -): cochar(\sA) \times cochar(\sK) \to \Z.
\]
A cocharacter $\sw$ of $\sK$ corresponds to a line bundle $\L_\sw$, which has some weight $\lambda$.
The pairing is defined via
\[
\chi_p(\nu, \sw) = \inner{\nu}{\lambda}.
\]
We extend this pairing to real cocharacters and get a map
\[
\chi_p: \Lie_\R(\sA) \times \Lie_\R(\sK) \to \R.
\]

\section{K-theoretic stable envelopes}

K-theoretic stable envelopes for $\sX$ are   improved versions of the classes of attracting subvarieties $\Attr_\sigma(p)$, motivated by Schubert cells for various Grassmannians and flag varieties. They depend on a choice of an equivariant chamber $\fC$, a generic fractional line bundle $\sw \in \Pic(\sX) \otimes \R$ (called {\it the slope}) and a polarization $T^{1/2}$. The stable envelope is a correspondence
\[
\Stab^{\sX, K}_{\fC, \sw} \in \sK(\sX^\sA \times \sX)
\]
Restricting the first argument to a class of a fixed point $p \in \sX^\sA$, we obtain a class in $\sK(\sX)$ which is called the stable envelope of $p$ and is denoted $\Stab^{\sX, K}_{\fC, \sw}(p)$.

The stable envelope is characterized by the following properties:
\begin{enumerate}
\item $\supp \Stab \subset \Attr^f$,
which means that the matrix of the correspondence is upper-triangular with respect to the ordering defined by $\fC$.
\item Near the diagonal,
\[
\Stab^{\sX, K}_{\fC, \sw} = \cO_{\Attr^f} \otimes \text{line bundle}.
\]
More precisely,
\[
\Stab^{\sX, K}_{\fC, \sw}(p)|_p = \sqrt{\frac{\det N_p^-}{\det N_p^{1/2}}} \cdot \bigwedge\nolimits^\bullet(N_p^{-, \vee})
\]

The prefactor here is included here for the following reason. After multiplication by $\sqrt{\det N_p^-}$ we get a symmetrized product $\ahat(N_p^{-})$ instead of $\bigwedge\nolimits^\bullet(N_p^-)$, which contains less irrelevant information. The second determinant $\sqrt{\det N_p^{1/2}}$ is included to ensure $\Stab$ does not involve half-integral characters of $\sA$ (but involves $\sqrt{\hbar}$, though) since
\[
\sqrt{\frac{\det N_p^-}{\det N_p^{1/2}}} = \frac{\hbar^{-\rk \ind_p/2}}{\det \ind_p}
\]

\item For $p_2 < p_1$:
\[
\deg_\sA \Stab^{\sX, K}_{\fC, \sw}(p_1)|_{p_2} + \chi_{p_1}(-, \sw) \subset \deg_\sA \Stab^{\sX, K}_{\fC, \sw}(p_2)|_{\times p_2} + \chi_{p_2}(-, \sw)
\]
\end{enumerate}
The third condition means that the Newton polygon for the $\sA$-character of the off-diagonal components of $\Stab$ should be contained in the corresponding diagonal component after a certain shift by a fractional character given by the slope.

\section{Equivariant elliptic cohomology}

For a smooth algebraic variety $\sX$ with an action of an algebraic torus $\sT$, let $\Ell_\sT(\sX)$ be the $\sT$-equivariant elliptic cohomology scheme.
It is a scheme over
\[
\Ell_\sT(pt) = \sT/q^{cochar(\sT)} \cong E^{\dim \sT}, 
\]
where $E = \C^\times/q^\Z$ is an elliptic curve.

Elliptic cohomology classes are sections of various line bundles over $\Ell_\sT(\sX)$. The elliptic case is slightly different from the ordinary equivariant cohomology and K-theory, where we would have affine schemes, and thus the classes are functions on them.

There are several natural ways to obtain elliptic cohomology classes. 
A vector bundle $\V$ over $\sX$ produces a section of a line bundle
\[
\vartheta(\V) \to \Ell_\sT(\sX),
\]
giving a homomorphism
\[
\vartheta: \ \sK_\sT(\sX) \to \Pic(\Ell_\sT(\sX))
\]
called the equivariant Thom class. Sometimes, where it can not cause confusion,  we will use the same notations for line bundles and their sections.

In order to introduce K\"ahler variables, and consier them on equal footing with equivariant variables, we will use the extended elliptic cohomology scheme
\[
\sE_\sT(\sX) = \Ell_\sT(\sX)\times E_{\Pic(\sX)},
\]
where
\[
E_{\Pic(\sX)} = \Pic(X)\otimes E
\]
for the same family of elliptic curves $E = \C^\times/q^\Z$.

\subsection{}
There is a natural Poincar\'e bundle $\cU$ over the extended elliptic cohomology scheme.
The Chern class is a map
\[
c: \Pic_\sT(\sX) \to Maps(\Ell_\sT(\sX) \to E).
\]
which can be viewed as 
\[
\tilde c: \Ell_\sT(\sX) \to E_{\Pic_\sT(\sX)}^\vee.
\]
The bundle $\cU$ is the pullback of the tautological bundle (see \ref{tautological_line_bundle}) under the map
\[
\tilde c \times \id: \Ell_\sT(\sX) \times E_{\Pic_\sT(\sX)} \to E_{\Pic_\sT(\sX)}^\vee \times E_{\Pic_\sT(\sX)}
\]

\subsection{} \label{universaldefined}
We will be focused on varieties with isolated fixed points. An elliptic cohomology class is uniquely defined by its restrictions to fixed points which are sections of bundles over the base variety $\sE_\sT(\sX)$. The equivariant coordinates we denote by $a$ and the K\"ahler coordinates by $z$ for simplicity.

In simple English, if the coordinate $z_i$ corresponds to a line bundle $\L_i$, the fixed point component of the bundle $\cU$ is
\[
\prod_i \frac{\vartheta(a^{\lambda_i} z_i)}{\vartheta(a^{\lambda_i})\vartheta(z_i)},
\]
where $\lambda_i$ is the weight of the bundle $\L_i$ at the fixed point.

\chapter{Elliptic stable envelopes}

\section{Elliptic curves and theta-functions}

Consider an elliptic curve
\[
E = \C^\times/q^\Z, \ \ q \in \C^\times.
\]
We will use a multiplicative coordinate $x$ on $\C^\times$.
The classical odd Jacobi theta-function is defined as
\[
\vartheta(x) = (x^{1/2} - x^{-1/2}) \prod_{i\geq 1} (1-q^i x)(1-q^i/x).
\]
It is a regularized version of a natural $q$-periodic product
\[
\prod_{i \in \Z} (1-q^i x).
\]
Because of the regularization, theta-function is not a function on $E$, but a holomorphic section of a line bundle of degree 1, with a  unique simple zero at $x=1$. In other words, the theta-function is a section of a trivial bundle over $\C^\times$ with automorphy factors (quasiperiods) corresponding to the difference equation
\[
\vartheta(qx) = -\frac{1}{\sqrt{q} x} \vartheta(x).
\]
Iterating it, we get
\[
\vartheta(q^n x) = (-1)^n x^{-n} q^{-n^2/2} \vartheta(x).
\]
Note that, strictly speaking, $\vartheta(x)$ is a double-valued section, since
\[
\vartheta(e^{2\pi i} x) = -\vartheta(x).
\]
It has also an expansion which will be useful for understanding behavior in various limits.
\begin{equation} \label{thetaadditive}
\vartheta(x) = \left(
\prod_{i\geq 1} (1-q^i)
\right)^{-1} \left(
\sum_{n \in \Z} (-1)^n  x^{n+1/2}q^{n(n+1)/2}
\right).
\end{equation}
For example, consider the limit
\[
\lim_{q\to 0} \vartheta(xq^\sw), \ \ \sw \in \R.
\]
In the expansion \ref{thetaadditive}, for integer $\sw$, two terms with $n = -\sw-1$ and $n = -\sw$ will be dominating, while for other $\sw$ only a single term $n = -\lfloor \sw \rfloor - 1$ will be dominating. Thus,
\begin{equation} \label{thetaasympt}
\vartheta(xq^\sw) \sim \begin{cases}
(-1)^{\intp{\sw}+1} q^{\mathbb{q}(\sw)} x^{-\lfloor \sw \rfloor - 1/2}, & \sw \not \in \Z, \\
(-1)^{\intp{\sw}+1} q^{\mathbb{q}(\sw)} x^{-\sw-1/2} (1- x), & \sw \in \Z.
\end{cases}
\end{equation}
where
\[
\mathbb{q}(\sw) = \frac{\intp{\sw}(\intp{\sw}+1)}{2} - \sw\left(
\intp{\sw}+\frac{1}{2}
\right)
\]
is a piecewise linear function satisfying the differential equation
\[
\frac{d^2}{d\sw^2}\mathbb{q}''(\sw) + \sum_{n\in \Z} \delta(\sw-n) = 0.
\]
As a consequence, we have
\[
\lim_{q\to 0} \frac{\vartheta(z a q^\sw)}{\vartheta(aq^\sw)} = \begin{cases}
z^{-\intp{\sw}-1/2}, & \sw \not \in \Z\\
z^{-\sw - 1/2} \cdot \frac{1-az}{1-a}, & \sw \in \Z.
\end{cases}
\]

There is a wonderful book on theta-functions \cite{mumf}. Up to some function of $q$, our theta-function is equal to $\vartheta_{11}(z|\tau)$, where
\[
z = e^{2\pi i z}, \ \ q = e^{2 \pi i \tau}.
\]

\section{Line bundles over abelian varieties}

Fix an elliptic curve
\[
E = \C^\times/q^\Z
\]
and consider line bundles over the abelian variety
\[
\cA = E_{a_1} \times ... \times E_{a_n}.
\]
In other words, we have an algebraic torus
\[
\sA \cong (\C^\times)^n
\]
with multiplicative coordinates $a_1, ..., a_n$ and the abelian variety can be regarded as
\[
\cA = \sA/q^{cochar(\sA)}
\]
The set $\Pic(\cA)$ of line bundles over $\cA$ is a disconnected algebraic variety, and let $\Pic_0(\cA)$ be its connected component containing the trivial bundle $\cO \to \cA$.
There is an exact sequence
\[
0 \to \Pic_0(\cA) \to \Pic(\cA) \to \NS(\cA) \to 0,
\]
and $\NS(\cA)$ is called the N\'eron-Severi group. We would like to describe elements of $\NS(\cA)$ explicitly.

\subsection{Tautological line bundle}\label{tautological_line_bundle}

The set $\Pic_0(\cA)$ is known to be the dual abelian variety to $\cA$. Let  $\sA^\vee$ be the dual torus with coordinates $\{z_i\}$ dual to $\{a_i\}$. On the product 
\[
\cA \times \cA^\vee = \sA/q^{cochar(\sA)} \times \sA^\vee/q^{cochar(\sA^\vee)}
\]
there is a line bundle, called the Poincar\'e bundle, with a section
\[
\prod_i \frac{\vartheta(a_iz_i)}{\vartheta(a_i) \vartheta(z_i)}
\]
such that points $(z_1,...,z_n)$ bijectively parametrize all line bundles in $\Pic_0(\cA)$. 

\subsection{}

On the set of sections there is a natural  action of the group of cocharacters:
\[
s(a_1, ..., a_n) \to s(a_1 q^{\sw_1}, ..., a_n q^{\sw_n}), \ \ \sw_i \in \Z.
\]
Explicitly, any rational section can be written as
\[
s(a) = \prod_i \vartheta(a^{\lambda_i} c_i)^{m_i},
\]
for some characters $\lambda_i \in char(\sA)$ and shifts $c_i \in E$.
It transforms as
\begin{equation} \label{ashift}
s(aq^\sw) = \prod_i \vartheta(q^{\inner{\lambda_i}{\sw}} a^{\lambda_i} c_i) = \prod_i (a^{\lambda_i} c_i)^{-m_i \inner{\lambda_i}{\sw}} q^{-m_i \inner{\lambda_i}{\sw}^2/2} \vartheta(a^{\lambda_i} c_i).
\end{equation}
The exponent of $q$ encodes the class of the bundle in $\NS(\cA)$ and is called the degree. It is a quadratic function on the cocharacter lattice:
\[
\deg s \in S^2 char(\sA): cochar(\sA) \to \Z,
\]
\[
\deg(s): \sw \mapsto -\sum_i m_i\frac{\inner{\lambda_i}{\sw}^2}{2}.
\]

\begin{example}
For example, consider a section
\[
s = \frac{\theta(a^2 z)\theta(a)}{\theta(a^2)\theta(az)}
\]
over $E_a \times E_z$.
Its degree is
\[
\deg s = -(2a+z)^2 -a^2 + (2a)^2 + (a+z)^2 = -2az,
\]
which is the same as for the Poincar\'e bundle
\[
\frac{\theta(az)}{\theta(a)\theta(z)}
\]
\end{example}

The following proposition is straightforward from \ref{ashift}.
\begin{proposition}
For any $\sw \in cochar(\sA)\otimes \R$ there is $\alpha \in \R$ such that the following limit exists
\[
\lim_{q\to 0} q^{\alpha} s(aq^\sw).
\]
\end{proposition}
For integer $\sw$ we can take $\alpha = -\inner{\deg s}{\sw \otimes \sw}$; for other $\sw$ the expression is more complicated involving floor functions.


\section{Limit from elliptic cohomology to K-theory}

In the limit $q\to 0$, the elliptic curve $E = \C^{\times}/q^\Z$ degenerates to $\C^\times$, and sections of line bundles over $E$ to certain multi-valued functions on $\C^\times$; for example
\[
\vartheta(x)  \to \ahat(x) = x^{1/2} - x^{-1/2}.
\]
Equivariant elliptic cohomology classes, which are sections of line bundles over $\Ell_\sT(\sX)$, degenerate to sections of line bundles over $\sK_\sT(\sX)$, which are double-valued functions. After restricting to fixed points, they become Laurent polynomials with, possibly, half-integer exponents.

We will be focused on varieties with isolated fixed points, where the computations can be done by localization. For each point we take the limit of the corresponding restriction of an elliptic class.

\section{Balanced sections}

Recall the abelian variety is 
\[
\cA = \sA/q^{cochar(\sA)}.
\]
We say that a section $s(a)$ is balanced (in variables $a$) if
for any fractional cocharacter $\sw \in cochar(\sA)\otimes \R$
in the limit $q\to 0$, the section $s(aq^\sw)$ has constant asymptotics:
\[
s(a q^\sw) \sim O(q^0).
\]
A necessary condition for being balanced is that its degree is linear. However, it is not sufficient, since it guarantees the constant asymptotics property only for integral cocharacters. A section of  degree 0 is called {\it numerically balanced}.

\begin{example}
The section
\[
\frac{\vartheta(a_1 a_2)}{\vartheta(a_1) \vartheta(a_2)}
\]
 is not balanced over $E_{a_1} \times  E_{a_2}$, but for any fixed $a_2$ it is balanced over $E_{a_1}$, and vice versa.
\end{example}
\begin{example} Consider a section
\[
s(a) = \frac{\vartheta(a)^4}{\vartheta(a^2)}.
\] 
The asymptotical behavior of $s(aq^\sw)$ can be easily computed using \ref{thetaasympt}
\[
s(aq^\sw) = \frac{\vartheta(aq^\sw)^4}{\vartheta(a^2 q^{2\sw})} \sim q^{4\qq(\sw) - \qq(2\sw)}.
\]
The graph of the function $4\qq(\sw) - \qq(2\sw)$ is shown on the picture.
\begin{figure}[!htb]
\begin{center}
\begin{tikzpicture} [scale=2]
\draw [->] (-2.5,0) -- (2.5,0) node [anchor=west] {$\sw$};
\draw [->] (0,-1) -- (0,0.5)  node [anchor=south west] {$4\qq(\sw)-\qq(2\sw)$};
\draw [thick] (-2,0) -- (-1.5,-0.5) -- (-1,0) -- (-0.5,-0.5) -- (0,0) -- (0.5,-0.5) -- (1,0) -- (1.5,-0.5) -- (2,0) ;
\foreach \i in {-2,-1,1,2} {
\node [anchor=south] at (\i,0)  {$\i$};
\draw [fill=black] (\i,0) circle(0.02);
}
\node [anchor=south west] at (0,0)  {$0$};
\draw [fill=black] (0,0) circle(0.02);
\draw [fill=black] (0,-0.5) circle(0.02) node [anchor=north west] {$-\frac{1}{2}$};
\end{tikzpicture}
\end{center}
\caption{The graph of the function $4\qq(\sw) - \qq(2\sw)$ showing the order if $s(a q^\sw)$ at $q \to 0$.}
\end{figure}
We see that for all $\sw \not \in \Z$ the limit is divergent. This is an example of a section which is numerically balanced but not balanced.
\end{example}

There is a simple description of balanced sections.
\begin{proposition}
A section $s(a,z)$ is balanced in $a$ if and only if it can be written as
\[
s(a,z) = \sum \prod_i \frac{\vartheta(a^{\lambda_i} ... )}{\vartheta(a^{\lambda_i} ... )}
\]
The right hand side means an expression which can be written as a polynomial of
\[
\frac{\vartheta(a^\lambda c)}{\vartheta(a^\lambda \overline c)}
\]
for all characters $\lambda \in char(\sA)$ and $c,\overline c$ not depending on $\sA$.
\end{proposition}
\begin{proof}
It is enough to consider factorized sections
\begin{equation} \label{section}
\prod_i \vartheta(a^{\lambda_i}c_i)^{m_i},
\end{equation}
We can assume it is the section of the smallest number of factors, which cannot be represented as in the proposition. It follows that there exists a cocharacter $\sw \in cochar(\sA) \otimes \Q$ so that the asymptotics for $\sw \pm \eps$ are different. Indeed, let $\lambda_1$ be the weight in \ref{section} such that other integer multiples $\{n \lambda_1: n \in \Z\}$ are not present. Then there we can choose $\sw$ such that $\inner{\lambda_1}{\sw} \in \Z$, and all others $\inner{\lambda_i}{\sw} \not \in \Z$.
\end{proof}

\section{Limits}

Now assume that we have two groups of variables, which means that the abelian variety is
\[
\cA = \sA/q^{cochar(\sA)} \times \sK/q^{cochar(\sK)}.
\]
for  tori $\sA$ and $\sK$. Coordinates on $\sA$ are denoted by $a$ and they will later be equivariant variables. Coordinates on $\sK$, denoted by $z$, will be K\"ahler variables later.

Automorphy factors of each section $s(a,z)$ define a pairing
\[
\chi: cochar(\sA) \times cochar(\sK) \to Z,
\]
such that
\[
s(aq^\sw, z) = q^\bigstar a^\bigstar z^{\chi(\sw, -)} s(a, z)
\]
and
\[
z(a, zq^\sv) = q^\bigstar z^\bigstar a^{\chi(-, \sv)} s(a, z)
\]
for any (integer!) cocharacters $\sw \in cochar(\sA)$, $\sv \in cochar(\sK)$.

If the section $s(a,z)$ is numerically balanced over $E_a$ and $E_z$, we do not have monomial factors with $\bigstar$.

The pairing $\chi(-, -)$ can be recovered from the cross-part of the degree:
\[
\deg s \subset S^2 \left(
char(\sA) \oplus char(\sK)
\right) \to  char(\sA) \otimes char(\sK).
\]
We are interested is the limit  when $q \to 0$ after a shift of $a$ by a fractional cocharacter.
\[
\sw \in cochar(\sA) \otimes \R \cong \Lie_\R(\sA)
\]
The following proposition is straightforward:
\begin{proposition}
For a section $s(a,z)$ and $\sw \in cochar(\sA) \otimes \R$ let $\alpha \in \R$ be such that the limit
\[
\lim_{q\to 0} q^\alpha s(a q^\sw, z) 
\]
exists. Then
\[
\lim_{q\to 0} q^\alpha s(a q^\sw, z) \in \C(a^{1/2}, z^{1/2}).
\]
\end{proposition}

Let us discuss the behavior of the limit as a function of $z$, when $z$ goes to any infinity of $\sK$. We have to restrict ourselves to the sections of a very special kind, which will suffice for our applications.

\begin{proposition}
Let the section $s(a,z)$ be balanced over $E_a$ and over $E_z$, and suppose that it can be represented as
\[
s(a,z) = \sum_i f_i(a) g_i(z) \prod_k \frac{\vartheta(a^{\lambda_{i,k}}z^{\mu_{i,k}} c_{i,k})}{\vartheta(a^{\lambda_{i,k}} \overline c_{i,k})\vartheta(z^{\mu_{i,k}} \overline{\overline{ c_{i,k}}})}
\]
for some characters $\lambda_i \in char(\sA), \mu_i \in char(\sK)$.
Let $\sw \in cochar(\sA)\otimes \R$, $\sv \in cochar(\sK) \otimes \R$.
Then the following limits are finite:
\[
\lim_{z \to 0} z^{-\chi(\sw, -)} \lim_{q\to 0} s(a q^\sw, z), \ \ \lim_{a \to 0} a^{-\chi(-, \sv)} \lim_{q\to 0}  s(a, zq^\sv).
\]
The limit $z \to 0$ means that $z$ goes to any infinity of the torus $\sK$.
\end{proposition}
\begin{proof}
Let us prove the statement about the $z$-limit. It is enough to prove the statement for a factorized section.
Let us start with the function
\[
s(a,z) = \frac{\vartheta(az)}{\vartheta(a)\vartheta(z)}.
\]
For integral shifts $\sw$ the statement is clear. For non-integer $\sw$, a simple computation shows that
\[
s(aq^\sw, z) \sim \frac{(az)^{-\intp\sw - 1/2}}{(a)^{-\intp\sw - 1/2}(z^{1/2}-z^{-1/2})}.
\]
The twist by $z^{-\chi(\sw, -)}$ is just multiplication by $z^\sw$. The resulting function is bounded for $z\to 0$ and $z\to \infty$, since
\[
-\frac{1}{2}\leq \sw - \intp{\sw} - \frac{1}{2} \leq \frac{1}{2}.
\]

In general, for the function
\[
s(a,z) = \sum_i f_i(a) g_i(z)  \frac{\vartheta(a^{\lambda_i}z^{\mu_i} c_i)}{\vartheta(a^{\lambda_i} \overline c_i)\vartheta(z^{\mu_i} \overline{\overline{ c_i}})},
\]
we get
\[
z^{\chi(\sw, -)} = \prod_i z^{\mu_i \inner{\lambda_i}{\mu_i}}.
\]
Twisting each factor  $i$-th by $z^{\mu_i \inner{\lambda_i}{\mu_i}}$,
we reduce to the case already considered.

\end{proof}

\begin{remark}
The bounding of the $z$-degree in the proof is similar to the window condition in the K-theoretic stable envelope. This observation will be developed in the context of 3d mirror symmetry, where $z$ plays a role of an equivariant parameter for the dual variety.
\end{remark}

\begin{proposition} \label{twopol}
Let $V$ be a representation of $\sA$, and $V_1, V_2$ be a two polarizations of $V$, that is, subrepresentations such that
\[
V_1 \oplus \hbar^{-1} V_1^\vee = V,
\]
\[
V_2 \oplus \hbar^{-1} V_2^\vee = V.
\]
Then the section
\[
s(a) = \frac{\vartheta(V_1)}{\vartheta(V_2)}
\]
is balanced in $\sA$.
\end{proposition}
\begin{proof}
Consider the virtual character
\[
W = V_1 \ominus V_2 \in K_\sA(pt),
\]
so that
\[
W + \hbar^{-1} W^\vee = 0.
\]

It follows that $W$ can be represented as
\[
W = \sum_i \left(
w_i - \frac{1}{\hbar w_i}
\right)
\]
for some characters $w_i \in char(\sA)$.
Then
\[
\frac{\vartheta(V_1)}{\vartheta(V_2)} = \vartheta(W) = \pm \prod \frac{\vartheta(w_i)}{\vartheta(\hbar w_i)}
\]
which is balanced.
\end{proof}

\section{Resonances}
Let $s(a,z)$ be a section balanced in $a$. A point $\sw \in \Lie_\R(\sA)$ is called a {\it resonance} of the function $s(a,z)$ if the following limit is a non-constant function of $a$:
\[
\lim_{q\to 0} s(a q^\sw, z) \not \in \C(z^{1/2}).
\]
For a collection of sections $s_i(a,z)$ balanced in $a$ we denote by $\Res\left(\{s_i(a,z)\}\right)$ the union of resonances of all functions in the collection.

Let $\Lambda^\vee \subset char(\sA) \subset \Lie_\R (\sA)^\vee$ be the sublattice generated by $\sA$-characters appearing in any $s_i(a,z)$. It is a free lattice, and let $\Lambda \subset cochar(\sA) \subset \Lie_\R(\sA)$ be the dual lattice.

The following proposition is clear.
\begin{proposition}
The set $\Res(\{s_i(a,z)\}$ is a $\Lambda$-periodic arrangement of hyperplanes in $\Lie_\R(\sA)$.
\end{proposition}
Explicitly, it is a subarrangement of the set
\[
\inner{\lambda}{\sw} \in \Z
\]
for all characters $\lambda \in char(\sA)$ appearing in $\{s_i(a,z)\}$. 

In what follows, resonances with respect to equivariant variables will be called {\it resonances}, and resonances with respect to K\"ahler variables will be called {\it walls}. 

\section{Elliptic stable envelopes}

Elliptic stable envelopes were defined by A.Okounkov and M.Aganagic in \cite{AOElliptic}.
They can be viewed as analytic continuation of the K-theoretic stable envelopes for complex slopes.

Elliptic stable envelopes $\Stab_{\fC, T^{1/2}}^{\sX}$ are defined using the same data as K-theoretic stable envelopes, except the slope. They are maps of line bundles over the extended elliptic cohomology scheme 
\[
\Theta(T^{1/2} \sX^A) \otimes \cU' \to \Theta(T^{1/2} \sX) \otimes \cU.
\]
satisfying the following two conditions:
\begin{enumerate}
\item The support is triangular with respect to $\fC$.
\item Near the diagonal they "look like" the classes of the attracting subvarieties:
\[
\Stab_{\fC, T^{1/2}}^{\sX}(p)|_p = \vartheta(N_{p,<0}).
\]
\end{enumerate}

Here $\cU$ is the universal bundle defined in \ref{universaldefined} for $\sX$, and $\cU'$ is the universal bundle for $\sX^\sA$ with a shift of K\"ahler parameters. Since $\sX^\sA$ is the union of isolated fixed points, it is just a shift of $z$ by certain power of $\hbar$ at each point. The shift is uniquely fixed by the normalization condition and should be equivalent to a change of the degree by
\[
\frac{\vartheta(N_p^-)}{\vartheta(N^{1/2}_p)}
\]
for a fixed point $p$.

The elliptic stable envelopes are always unique, but their existence is proven in \cite{AOElliptic} only for Nakajima quiver varieties. However, later they have been constructed  in a very high generality, see \cite{OkInductive1, OkInductive2}.

\section{Infinitesimal slopes}

The walls of $\sX$ are defined in \ref{walls_def}, and they partition $\Lie_\R(\sK)$ into K\"ahler chambers. Certain chambers in $\Lie_\R(\sK)$ are very important from many perspectives: the quantum difference equations, Bethe ansatz, abelianization techniques, shuffle algebras.

Let $\cU_0 \subset H^2(\sX, \R)$ be an infinitesimal neighborhood of 0, and $\Walls_0(\sX) \subset \Walls(\sX)$ be the set of hyperplanes passing through 0. They partition $\cU_0$ into chambers
\[
\cU_0 \setminus \Walls_0(\sX) = \coprod \fD_i(X),
\]
and we denote by $\fD_+(\sX), \fD_-(\sX)$ the chambers of ample and anti-ample bundles respectively.
Let $\fD$ be any chamber in $\cU_0 \setminus \Walls_0(\sX)$.

\section{G-fixed subvarieties}

Let $G = \nu_\sw = \langle e^{2\pi i \sw } \rangle \subset \sA$ be a finite cyclic subgroup.  The subvariety $X^{G}\subset X$ of $G$-fixed points has the same set of $\sA$-fixed points as $\sX$. 
An equivariant chamber $\fC \subset \Lie_\R(\sA)$ defines the equivariant chamber for $X^G$, which we denote by the same symbol.

We have a canonical map
\[
\Pic(\sX) \to \Pic(\sX^G),
\]
and thus a K\"ahler chamber $\fD$ for $\sX$ defines a chamber for $\sX^G$. Note that not any chamber of $\sX^G$ can be obtained this way.

The polarization for $\sX^G$ is defined as the $G$-fixed part $T^{1/2,G}$ of the polarization $T^{1/2}$ of $\sX$.

It turns out that K-theoretic stable envelopes of infinitesimal slopes for $\sX^G$ can be obtained from elliptic stable envelopes for $\sX$ in a beautiful way. For this, we need to discuss equivariant limits of elliptic stable envelopes.

\section{Equivariant limits of elliptic stable envelopes}

Let $\sw \in \Pic(\sX) \otimes \R$ be generic.
It was shown by M.Aganagic and A.Okounkov that the limit
\[
\lim_{q \to 0} \Stab(a, zq^\sw)
\]
is equal to the K-theoretic stable envelope for the slope $\sw \in \Pic(X) \otimes \R$. In this case, the stable envelopes are balanced in $z$, so that the limit is well defined.

We need to consider limits with shifts of the equivariant parameters. 
The stable envelopes are not balanced with respect to the equivariant parameters. However, if we divide them by the polarization, they become balanced.

Let us define
\[
\stab(p) = \frac{\Stab(p)}{\vartheta(T^{1/2})}, \ \ \ \stab(p)|_r = \frac{\Stab(p)|_r}{\vartheta(T^{1/2}|_r)}.
\]
The formula for $\stab(p)$ is to be understood in the following sense:  if we get a trivial weight in the numerator or the denominator, we should remove the corresponding factor (note that this does not change the quasiperiods).
\begin{proposition}
The classes $\stab(p)$
are balanced in equivariant variables.
\end{proposition}
\begin{proof}
From the definition of the stable envelopes, it is clear that $\stab(p)$ is a class of degree 0. For hypertoric varieties, it follows from the explicit construction that $\stab(p)$ is balanced. The stable envelopes for the Nakajima quiver varieties can be constructed from the stable envelopes for the hypertoric varieties using the abelianization procedure \cite{AOElliptic}, and thus by uniqueness they are balanced.
\end{proof}

After we divided by $\vartheta(T^{1/2})$, we no longer get an integral elliptic class. Luckily, the denominator is very simple and can be controlled in the limit.

\begin{theorem} \label{mainequiv}
Consider the following class:
\[
S = \bigwedge\nolimits^\bullet(T^{1/2, G})^\vee \cdot \left.
\stab(p)
\right|_{a \to a q^\sw}
\]
Then in the following limit we obtain the K-theoretic stable envelopes for the $G$-fixed subvariety:
\[
\lim_{z \to 0_{\fD}} \left(
z^{\chi(\sw, -) - \chi_p(\sw, -)} 
\lim_{q\to 0} S
\right) = \pm \hbar^{-\intp{\ind_p \cdot \sw}- \rk \ind_p/2 + \rk \ind_p^G/2} \cdot \Stab_{\sX^{G}, \fC, T^{1/2, G}}
\]
\end{theorem}
Here the left hand side gives an element of $K_\sT(\sX)$, and the right hand side is an element of $K_\sT(\sX^G)$ which we identify
using the restriction
\[
K_\sT(\sX) \to K_\sT(\sX^G).
\]
The $z$-limit is taken for any infinitesimal K\"ahler chamber, and we get K-theoretic stable envelopes for $\sX^G$ for the corresponding infinitesimal chamber.

Before we prove the theorem, let us figure out what we get near the diagonal.
Let us start with a limit without a shift.
Diagonal components of elliptic stable envelopes are
\[
\Stab(p)|_p = \vartheta(N_p^-)
\]
In the limit $q \to 0$ we get diagonal components of K-theoretic stable envelopes in the correct normalization: 
\[
\lim_{q \to 0}
\left(
\bigwedge\nolimits^\bullet(T^{1/2, G})^\vee \cdot  \frac{\vartheta(N_p^-)}{\vartheta(N^{1/2})}
\right)
 = \sqrt{\frac{\det N_p^-}{\det N_p^{1/2}}} \cdot \bigwedge\nolimits^\bullet (N_p^-)^\vee = \frac{\hbar^{-\rk \ind_p/2}}{\det \ind_p} \cdot  \bigwedge\nolimits^\bullet (N_p^-)^\vee.
\]

Now consider the case with a nontrivial shift $\sw$.
The main idea is the following: if 
\[
V = \sum w_i
\]
is a representation of $\sT$, then for certain $\alpha$ the limit
\[
\lim_{q \to 0} q^\alpha \left.
\vartheta(V)
\right|_{a\to aq^\sw} 
\]
is equal to
\[
\bigwedge\nolimits^\bullet(V^{G}) = \prod_{i: \inner{w_i}{\sw} \in \Z} (1-w_i).
\]
up to a monomial prefactor. Indeed, in the symptotics, the factors with integer shifts give rise to $\ahat$-type factor, whicle the other give monomial contribution.

\begin{lemma}
\begin{multline}
\lim_{q \to 0} \left[
\frac{\vartheta(N_p^-)}{\vartheta(N^{1/2}_p)}
\right]_{a \to aq^\sw} = \pm \hbar^{-\intp{{\ind_p}\cdot{\sw}} - \rk \ind_p/2 + \rk \ind_p^G/2} \cdot \frac{\ahat(N_p^{-, G})}{\ahat(N_p^{1/2, G})} =
\\
=
\pm \frac{\hbar^{-\intp{{\ind_p}\cdot{\sw}}  - \rk \ind_p/2}}{\det \ind^G} \cdot \frac{\bigwedge\nolimits^\bullet(N_p^{-, G})^\vee}{\bigwedge\nolimits^\bullet(N_p^{1/2, G})^\vee} 
\end{multline}
\end{lemma}
\begin{proof}
We have
\[
\frac{\vartheta(N_p^-)}{\vartheta(N_p^{1/2})} = \frac{\vartheta(\hbar \cdot \ind_p)}{\vartheta(\ind_p)}.
\]
Letting
\[
\ind_p = w_1 + ... + w_r,
\]
then
\[
\left.
\frac{\vartheta(N_p^-)}{\vartheta(N^{1/2}_p)}
\right|_{a = aq^\sw} =  \prod_i \frac{\vartheta(\hbar w_i q^{\inner{w_i}{\sw}})}{\vartheta( w_i q^{\inner{w_i}{\sw}})}
 = 
\hbar^{-\intp{\inner{\ind_p}{\sw}} - \rk \ind_p/2} \cdot \prod_{i:\inner{w_i}{\sw} \in \Z} \frac{1-\hbar w_i}{1-w_i}.
\]
On the other hand,
\[
\frac{\ahat(N_p^{-, G})}{\ahat(N_p^{1/2, G})} = \frac{\ahat(\hbar \cdot\ind_p^{ G})}{\ahat(\ind_p^{G})} = \hbar^{-\rk \ind^G_p/2}
\prod_{i:\inner{w_i}{\sw} \in \Z} \frac{1-\hbar w_i}{1-w_i}.
\]
\end{proof}

As a corollary, we get that the limit of diagonal components of $S$ after the shift are
\[
\lim_{q\to 0} S|_p = \frac{ \hbar^{-\intp{\ind_p\cdot\sw} - \rk \ind_p/2 } }
{\det \ind_p^G}
\bigwedge\nolimits^\bullet(N_p^{-,G})^\vee,
\]
while the K-theoretic stable envelopes for $\sX^G$ are normalized as
\[
\Stab^{\sX^G}(p)|_p = \frac{ \hbar^{ - \rk \ind^G_p/2 } }
{\det \ind_p^G}
\bigwedge\nolimits^\bullet(N_p^{-,G})^\vee
\]
So, we proved that on the diagonal we get the correct values.

Let us now proceed to the proof of Theorem \ref{mainequiv}.
\begin{proof} (of Theorem \ref{mainequiv})

We see that $S$ is an integral K-theoretic class.
The K-theoretic support condition for $\sX^G$ follows from the support condition for $\sX$ since, the ordering of the fixed points is the same.

The last thing to check is the Newton polytope inclusion of the off-diagonal components into the diagonal components.
Applying the proposition 4.3 of \cite{AOElliptic}, we get that
\begin{multline}
\deg_\sA 
\lim_{z \to 0_\fD} \left(
\bigwedge\nolimits^\bullet (T^{1/2,G}_r)^\vee \otimes
\lim_{q \to 0} \stab(p)|_r
\right) \subset
\deg_\sA 
\lim_{z \to 0_\fD} \left(
\bigwedge\nolimits^\bullet (T^{1/2,G}_r)^\vee \otimes
\lim_{q \to 0} \stab(r)|_r
\right) +\\
+ \text{infinitesimal shift}
\end{multline}
The infinitesimal shift is given by the chamber $\fD$. Originally, they were shifting $z$, and then taking the limit in $q$, but for infinitesimal slopes $\sw$ it is possible to take the limit without a shift, and then take the limit with respect to $z$, since $zq^\sw$ is "slowly changing" when $q\to 0$.

\end{proof}

For computational purposes, it is convenient to reformulate the theorem as follows. Let us consider the following normalization:
\[
\tilde T_{pr} = \frac{\Stab(p)|_r}{\Stab(r)|_r}
\]
It is a triangular matrix with $T_{pp}(a,z) = 0$ on the diagonal.

Let
\begin{equation}\label{anormal}
A_{pr}(a,\hbar) = \frac{\Stab^{K}_{\sX^G, \fD}(p)|_{r}}{\Stab^{K}_{\sX^G, \fD}(r)|_{r}}
\end{equation}
be the matrix of K-theoretic stable envelopes of $\sX^G$ with the similar normalization. Then the  theorem \ref{mainequiv} can be reformulated as:
\begin{theorem} We have \label{limitnormal}
\[
\lim_{z\to 0_\fD} Z \left(
\lim_{q\to 0} T(aq^\sw, z) 
\right) Z^{-1} = H A(a, \hbar) H^{-1}
\]
where $Z$ is the diagonal matrix
\[
Z = \diag(z^{\chi_p(\sw, -)})|_{p \in \sX^\sA},
\]
and $H$ is a diagonal matrix of monomials in $\hbar$ with signs
\[
H = \diag\left(
 \pm \hbar^{-m_p(\sw)/2}
\right)|_{p \in \sX^\sA}
\]
for
\[
m_p(\sw) = \inner{\det T_p^{1/2}}{\sw} - \intp{N_p^- \cdot \sw} + \frac{\rk \ind_p - \rk \ind_p^G}{2}.
\]
\end{theorem}

\section{Resonances of elliptic stable envelopes} \label{walls_def}

As an application of \ref{mainequiv}, we get a simple geometric description of the resonances of elliptic stable envelopes.

For an element $\sw \in \Lie_\R(\sA)$ consider the cyclic subgroup $G = \langle e^{2\pi i \sw} \rangle$. For generic $\sw$ the $G$-fixed subvariety is equal to $\sX^\sA$. Let us denote
\[
\Res(\sX) = \{
\sw \in \Lie_\R(\sA): X^{G} \neq X^\sA
\}.
\]
We call $\Res(\sX) \subset \Lie_\R(\sA)$ resonances. Earlier we defined resonances of a collection of functions. The connection between them is given by the following proposition.
\begin{proposition}
The following three hyperplane arrangements in $\Lie_\R(\sA)$ coincide:
\begin{enumerate}
\item $S_1 = \Res(X)$
\item $S_2 = \Res(\{\tilde T_{pr}(a,z)\}_{p,r \in \sX^\sA})$
\item $S_3 = \{\sw \in \Lie_\R(\sA): \inner{\alpha}{\sw} \in \Z, \ \alpha \in char_\sA(T_p\sX), p \in \sX^\sA) \}$.
\end{enumerate}
\end{proposition}
\begin{proof}
See \cite{KononovSmirnov1}.
\end{proof}

Resonances of $\{\tilde T_{pr}(a,z)\}_{p,r \in \sX^\sA}$ with respect to $z$-variables are called the walls:
\[
\Walls(\sX) = \Res_{z} \left(\{\tilde T_{pr}(a,z)\}_{p,r \in \sX^\sA}\right).
\]
Unfortunately, the walls do not have such a simple description in terms of the geometry of $\sX$, but they do have a description in terms of the symplectic dual variety $\sX^!$, see \ref{res_walls_duality}.

\section{Application: Hilbert scheme}

For an introduction to Nakajima quiver varieties and the Hilbert scheme of $n$ points on $\C^2$, see chapter \ref{explicitexamples}. Set-theoretically, the variety $\Hilb(\C^2, n)$ is the set of ideals $I \subset \C[x,y]$ of codimension $n$. There is an action of two-dimensional torus $\sT = (\C^\times)^2_{t_1t_2}$ via
\[
x \mapsto t_1^{-1} x, \ \ y \mapsto t_2^{-1} y.
\]
The torus $\sT$ has a factorization
\[
\sT = \sA_a \times \C^\times_\hbar, \ \ \ t_1 = a\sqrt\hbar, \  t_2 = a^{-1} \sqrt\hbar.
\]
The fixed points correspond to monomial ideals, and can be identified with Young diagrams with $n$ boxes as illustrated in Figure \ref{FP}.

\begin{figure}[h]
    \centering
    \begin{tikzpicture} [scale=1.0]
    \draw [very thick,  <->] (6,0) -- (0,0) -- (0,-6);
    \draw [thick, fill=yellow!10] (5,0)--(5,-1)--(3,-1)--(3,-2)--(2,-2)--(2,-4)--(1,-4)--(1,-5)--(0,-5)--(0,0)--cycle ;
    \draw (0,-1)--(3,-1);
    \draw(0,-2)--(2,-2);
    \draw(0,-3)--(2,-3);
    \draw(0,-4)--(2,-4);
    \draw(1,0)--(1,-4);
    \draw(2,0)--(2,-2);
    \draw(3,0)--(3,-1);
    \draw(4,0)--(4,-1);
    \node at (0.5,-0.5) {1};
    \node at (1.5,-0.5) {$x$};
    \node at (2.5,-0.5) {$x^2$};
    \node at (3.5,-0.5) {$x^3$};
    \node at (4.5,-0.5) {$x^4$};
    \node at (0.5,-1.5) {$y$};
    \node at (1.5,-1.5) {$xy$};
    \node at (2.5,-1.5) {$x^2y$};
    \node at (0.5,-2.5) {$y^2$};
    \node at (1.5,-2.5) {$xy^2$};
    \node at (0.5,-3.5) {$y^3$};
    \node at (1.5,-3.5) {$xy^3$};
    \node at (0.5,-4.5) {$y^4$};
    \end{tikzpicture}
    \caption{A fixed points of $\Hilb(\C^2)$. The diagram corresponds to the ideal generated by $x^5, x^3y, x^2y^2, xy^4, y^5$.}
    \label{FP}
\end{figure}
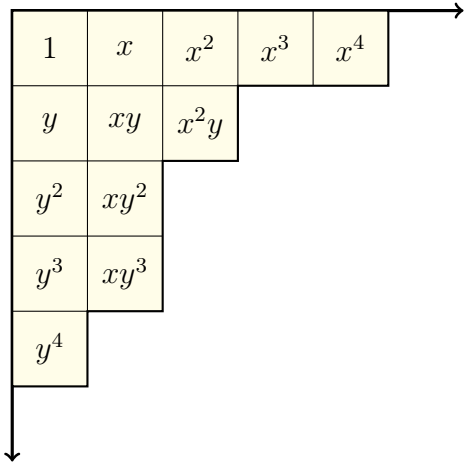

Let us describe the equivariant roots of the Hilbert scheme. For a Young diagram $\lambda$, the tangent space has the following description
\[
T_\lambda = \sum_{\square\in \lambda} (
t_1^{a(\square)+1} t_2^{-l(\square)} + t_1^{-a(\square)} t_2^{l(\square)+1}
),
\]
where $a(\square), l(\square)$ are the arm and leg lengths of the box, and the hook length is defined as
\[
hook(\square) = a(\square) + 1 + l(\square).
\]
Restricting to the subtorus $\sA \subset \sT$,
we get
\[
T_\lambda = \sum_{\square \in \lambda} (a^{hook(\square)} + a^{-hook(\square)})
\]
Thus, if we identify $\Lie_\sR(\sK) \cong \R$ the equivariant roots are
\[
\left\{\frac{a}{b} \in \Q: |b|<n\right\}
\]
For these values for $\sw$, the limit $\lim\limits_{q\to 0} \tilde T_{\lambda\mu}(aq^\sw, z)$ is nontrivial and is conjugate to the stable envelopes for the $\nu_b$-fixed subvariety for the cyclic subgroup
\[
\nu_b = \{w^k, k=0, ..., b-1\} \subset \sA \cong \C^\times.
\]

The $\nu_b$-fixed subvariety of $\sX$ corresponds to the cyclic quiver with $b$ vertices shown in Figure \ref{cyclic_quiver}. Indeed: on the space associated with the vertex, there should be a $\Z/b\Z$-grading, compatible with the arrow morphisms, see detailed explanation in \cite{NakajimaLectures1, NakajimaLectures2}.

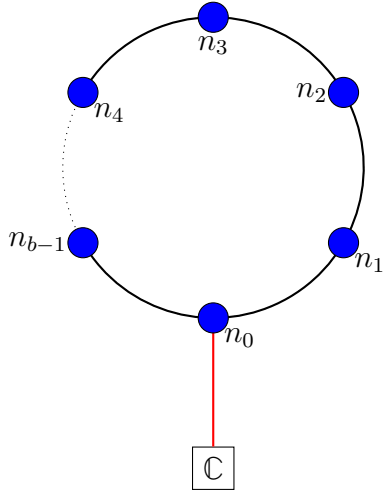
\begin{figure}[h]
    \centering
    \begin{tikzpicture}[scale=2]
    \draw [dotted] (0,0) circle(1);
    \node at (0,-2) [rectangle,draw] (C) {$\C$};
    \draw [red, thick, ->] (C) -- (0,-1);
    \draw [thick] (-.866,-.5) arc (210:510:1);
    \draw [fill=blue] (0,-1) circle(0.1) node[anchor=north west] {$n_0$};
    \draw [fill=blue] (0,1) circle(0.1) node[below=3pt ] {$n_3$};
    \draw [fill=blue] (-.866,-.5) circle(0.1) node[left=2pt] {$n_{b-1}$};
    \draw [fill=blue] (-.866,.5) circle(0.1) node[anchor=north west] {$n_4$};
    \draw [fill=blue] (.866,-.5) circle(0.1) node[anchor=north west] {$n_1$};
    \draw [fill=blue] (.866,.5) circle(0.1) node[left=2pt] {$n_2$};
    \end{tikzpicture}
    \caption{The cyclic quiver with $b$ vertices of dimensions $n_0, ..., n_{b-1}$ and 1-dimensional framing.}
    \label{cyclic_quiver}
\end{figure}

For the Hilbert scheme there are two infinitesimal slopes $\fD_+$ and $\fD_-$. The K-theoretic stable envelopes for these slopes are the same (surprisingly, there is no transition at 0). The slopes $\fD_{\pm}$ correspond to two distinguished infinitesimal slopes for $X^{\nu_b}$, and there is non-trivial transition between them.
Let $K^\pm(a, \hbar)$ be the corresponding K-theoretic stable envelopes for $X^{\nu_b}$, normalized as in \ref{anormal}.
Then our result \ref{limitnormal} in this particular case implies the following theorem:
\begin{theorem}
Let $\sw = a/b \in \Q$, where (a,b) are relatively coprime integers. Then
\[
\lim_{z\to 0} Z \left(
\lim_{q\to 0} \tilde T(aq^\sw, z)
\right) Z^{-1} = H \tilde K^+(a,\hbar) H^{-1},
\]
\[
\lim_{z\to \infty} Z \left(
\lim_{q\to 0} \tilde T(aq^\sw, z)
\right) Z^{-1} = H \tilde K^-(a,\hbar) H^{-1},
\]
where $Z$, $H$ are certain diagonal matrices in $z$ and $\hbar$, respectively (in the fixed point bases).
\end{theorem}

However, there are more walls passing through 0 for $X^{\nu_b}$.

According to \cite{SchifVas}, on the K-theory of the Hilbert scheme $\sX = \Hilb(\C^2)$, there is an action of the quantum toroidal algebra $\cU_\hbar(\widehat{\widehat{\mathfrak{gl}}}_1)$, also known as the Ding-Iohara-Miki algebra.
On the K-theory of $\sX^{\nu_b}$ there is an action of the algebra $\cU_\hbar(\widehat{\widehat{\mathfrak{gl}}}_b)$, and there is a beautiful interplay between them for different $b$. The results of the next chapter can be viewed as an attempt to understand this deep phenomenon in a more general framework.

\chapter{3d Mirror Symmetry} \label{3dmir}

\section{Data of 3d mirror symmetry}

Suppose we have two symplectic resolutions of singularities $\sX$ and $\sX^!$. The geometric objects related to each of them are shown in the following table:

\begin{figure}[h]
\begin{center}
\begin{tabular}{|c|c|c|}
\hline 
variety & $\sX$ & $\sX^!$ \\ 
\hline
{\bf Fundamental data:} & & \\
\hline 
torus & $\sT$ & $\sT^!$ \\ 
\hline 
equivariant torus & $\sA$ & $\sA^!$ \\ 
\hline
K\"ahler torus & $\sK$ & $\sK^!$ \\ 
\hline
weight of the symplectic form & $\hbar$ & $\hbar^!$ \\ 
\hline
stability parameter & $\theta$ & $\theta^!$ \\ 
\hline 
polarization & $T^{\sX, 1/2}$ & $T^{\sX^!, 1/2}$ \\ 
\hline
{\bf Additional data:} & & \\
\hline
Equivariant chamber & $\sigma $ & $\sigma^!$ \\
\hline
Elliptic stable envelopes & $\Stab_\sigma^{\sX} $ & $\Stab_{\sigma^!}^{\sX^!}$ \\
\hline
Quantum difference operators & $M^\sX(a,z)$ & $M^{\sX^!}(a,z)$ \\
\hline 
Shift operators & $S^\sX(a,z)$ & $S^{\sX^!}(a,z)$ \\
\hline 
\end{tabular}
\end{center}
\caption{Attributes of symplectic dual varieties $\sX$ and $\sX^!$.}
\end{figure}

{\it After we choose this data, we do not have freedom to choose equivariant chambers. They are canonically constructed from $\theta$ and $\theta^!$.}

For simplicity, we assume that the sets of fixed points of both varieties are finite:
\[
|\sX^\sA|, |(\sX^!)^{\sA^!}| < \infty.
\]

\section{Definition of 3d mirror symmetry}
The varieties $\sX$ and $\sX^!$ are said to be dual with respect to 3d mirror symmetry if the following three conditions are satisfied.
\begin{enumerate}
\item {\it There are isomorphisms
\[
\kappa: \sA \to \sK^!, \ \ \sK \to \sA^!, \ \ \C_\hbar^\times \to \C_{\hbar^!}^\times.
\]
In other words, 3d mirror symmetry exchanges equivariant and K\"ahler parameters.
}

\hspace{1cm}
Under these isomorphism, we can construct canonical elements in the equivariant Lie algebras, corresponding to stability parameters:
\[
\sigma = (d\kappa)^{-1}(\theta^!) \in \Lie_\R(\sA), \ \ \sigma^! = (d\kappa)(\theta) \in \Lie_\R(\sA^!).
\]
The cocharacters $\sigma, \sigma^!$ define decompositions of tangent spaces as fixed points into attracting and repelling subspaces.
As usual, we define attracting subsets and orderings on the sets of fixed points for $\sX$, $\sX^!$.
At least, in the case of Nakajima quiver varieties, there are elliptic stable envelopes
\[
\Stab_\sigma^\sX  \ \ \text{and} \ \ \Stab_{\sigma^!}^{\sX^!}
\]
associated with $(\sX, \sigma, T^{\sX, 1/2})$ and $(\sX^!, \sigma^!, T^{\sX^!, 1/2})$,
but we expect that they can be constructed in higher generality (for example, bow varieties \cite{NakajimaBows}).

\item {\it There is a bijection between the sets of fixed points
\[
\sX^\sA \to (\sX^!)^{\sA^!}, \ \ p \to p^!,
\]
which reverses the partial ordering.
}

Before we formulate the third condition, let us introduce a different normalization of elliptic stable envelopes which will allow us to achieve more symmetry
\[
\Stabb^\sX_\sigma(p) = \vartheta(N_{p^!}^-)\cdot \Stab^\sX_\sigma(p).
\]
Explicitly,
\[
\vartheta(N_{p^!}^-) = \prod_{w \in T_{p^!} \sX^!: \sigma^!(w)<0} \vartheta(a^w)
\]
is the product of $T^!$-characters of the tangent space of $\sX^!$ at $p^!$ which have negative pairings with the cocharacter $\sigma^!$. 
Similarly, we do the same for the dual side
\[
\Stabb^{\sX^!}_{\sigma^!}(p^!) = \vartheta(N_{p}^-)\cdot \Stab^{\sX^!}_{\sigma^!}(p^!).
\]
Note that each of $\Stabb^\sX$ and $\Stabb^{\sX^!}$ requires knowing of geometry of both varieties $\sX$ and $\sX^!$.
This definition is motivated by the observation that is makes them symmetric on the diagonal:
\[
\left.\Stabb^{\sX}_{\sigma}(p)\right|_{p}
=
\left.\Stabb^{\sX^!}_{\sigma^!}(p^!)\right|_{p^!} = \vartheta(N_p^-) \cdot \vartheta(N_{p^!}^-).
\]

The third condition for the 3d mirror symmetry is that the elliptic stable envelopes glue not only on the diagonal, but globally to an elliptic class on $\sX \times \sX^!$.

This next axiom is the most important and has been studied in \cite{AOprep, AOF}›
\item {\it
There exists a bundle $\mathfrak{M}$ over $\Ell_{\sT\times \sT^!}{\sX \times \sX^!}$ and a section $\mathfrak{m}$ such that
\[
\mathfrak{m}|_{p!} = \Stabb_\sigma^X(p), \ \mathfrak{m}|_{p} = \Stabb_{\sigma^!}^{X^!}(p!).
\]
}
\end{enumerate}

Let us consider the matrix of normalized elliptic stable envelopes,
\[
\tilde T^{\sX}_{pr} = \frac{\Stab_\sigma^\sX(p)|_r}{\vartheta(N_r^-)} = \frac{\Stabb_\sigma^\sX(p)|_r}{\vartheta(N_r^-) \vartheta(N_{p^!}^-)},
\]
and similarly $\tilde T^{\sX^!}_{pr}$ so that
\[
\tilde T_{pp} = 1, \ \ \tilde T^{\sX^!}_{pp} = 1. 
\]
Then, as a corollary we get
\[
\tilde T^\sX_{pr} = \kappa^* \tilde T^{\sX^!}_{r^! p^!},
\]
which means that elliptic stable envelopes for $\sX$ and $\sX^!$ coincide after identification by $\kappa$ which exchanges equivariant and K\"ahler variables. Now it is clear why we require that the bijection $\cdot^!$ reverses the orderings of fixed points: $\tilde T_{pr}^\sX$ can be nonzero only for $p>r$, while $\tilde T_{r^! p^!}^{\sX^!}$ only for $r^!>p^!$.

Note that the matrix elements $\tilde T_{pr}^{\sX}$ are balanced over $a$ as well as over $z$, which allows us consider various limits with shifts of $a$ and $z$ on equal footing. From the definition of elliptic stable envelopes, it follows that for any $\sv \in cochar(\sK)$ and $\sw \in cochar(\sA)$ we have
\[
\tilde T_{pr}^{\sX} (a, zq^\sv) = a^{\chi_p(-, \sv) - \chi_r(-, \sv)} T_{pr}^{\sX}(z,a)
\]
\[
\tilde T_{pr}^{\sX} (aq^\sw, z) = z^{\chi_p(\sw, -) - \chi_r(\sw, -)} T_{pr}^{\sX}(z,a)
\]
In other words, shifting $z$ by an integral cocharacter corresponds to a conjugation by a diagonal matrix depending on $a$ related to the classical multiplication by a line bundle. Shift of $a$ by an integral cocharacter corresponds to a conjugation by a diagonal matrix depending on $z$, which is related to the classical multiplication for the dual variety.

\section{Duality of walls and resonances} \label{res_walls_duality}

For elliptic stable envelopes, non-trivial limits depending both on equivariant and K\"ahler parameters appear if either the equivariant parameters are shifted by the elements on the resonant locus, or  the K\"ahler parameters are shifted by the elements on the walls. We have the following duality between them
\begin{proposition}
Under the isomorphism $\kappa$, 3d mirror symmetry switches the resonances with the walls:
\[
\Res(\sX) = \Walls(\sX^!), \ \ \Res(\sX^!) = \Walls(\sX).
\]
\end{proposition}

Note that the limit of elliptic stable envelopes to a wall can be dependent on K\"ahler parameters, however, the corresponding wall operator can be trivial. This is related with the fact that the corresponding matrix has the same limits as $z\to 0, \infty$. The simplest example of such phenomenon is the Hilbert scheme for which K-theoretic stable envelopes do not change of the K\"ahler parameter crosses integer walls.

\section{K-theoretic limits}

The result \cite{AOElliptic} of M.Aganagic and A.Okounkov, which was already used in the proof of \ref{mainequiv}, can be summarized as the following theorem:
\begin{theorem} For $\sw \in H^2(\sX, \R)$
\begin{enumerate}
\item The limit
$\lim\limits_{q\to 0} T_{pr}(zq^\sw, a)$ is a locally constant function on the complement to a certain hyperplane arrangement
\[
\Walls(\sX) \subset H^2(\sX, \R).
\]
\item For $\sw \in H^2(\sX, \R) \setminus \Walls(\sX)$, the limit equals
\[
\lim_{q\to 0} T_{pr}(zq^\sw, a) = A^{[\sw], X}_{pr},
\]
where $A^{[\sw], X}_{pr}$ is the matrix of K-theoretic stable envelopes with the slope $\sw$:
\[
A^{[\sw], X}_{pr} = \frac{\Stab^{\sw, X, K}_\sigma(p)|_r}{\Stab^{\sw, X, K}_\sigma(r)|_r}
\]
\end{enumerate}
\end{theorem}

In other words, the space $\Lie_\R(\sK)$ is divided into chambers, and the limit is different for each chamber. 
It is very natural to ask what the limits into walls are. It turns out that such limits depend on the K\"ahler paramters as well, and this dependence interpolates the transition between chambers: the limits at chambers can be obtained as further limits of the wall limit.

The answer can be formulated elegantly if we compare a limit to a non-regular slope with  limits to regular slopes in a small neighborhood, as illustrated in Figure \ref{walls}.

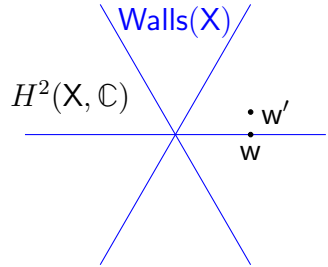
\begin{figure}[h]
\begin{center}
\begin{tikzpicture}
\draw[blue] (-1,-1.732) -- (1,1.732);
\draw[blue] (1,-1.732) -- (-1,1.732);
\draw[blue] (-2,0) -- (2,0);
\draw [fill=black] (1,0) circle(0.03) node[anchor=north] {$\sw$};
\draw [fill=black] (1,0.3) circle(0.03) node[anchor=west] {$\sw'$};
\draw [blue] (0,1.5) node {$\Walls(\sX)$};
\node at (-1.4,.5) {$H^2(\sX, \C)$};
\end{tikzpicture}
\end{center}
\caption{The theorem relates the limits of the elliptic stable envelope to non-regular slope $\sw$ and regular slope $\sw'$: the transition between them is given by a matrix which depends non-trivially on the K\"ahler parameters}
\label{walls}
\end{figure}

\begin{theorem} \label{mainfactorization}
Let $\sw \in H^2(\sX, \R)$ and $\sw' \not \in \Walls(\sX)$ be a regular slope in an infinitesimal neighbourhood of $\sw$. Then the limit has a factorization
\[
\lim_{q\to 0} T(zq^\sw, a) = Z \cdot  A^{[\sw'], X}
\]
Here $Z$ is a matrix which depends significantly only on K\"ahler variables, namely
\[
Z = \L_{\sw} Z'(z, \hbar) \L_{\sw}^{-1},
\]
where $\L_{\sw}$ is the operator of classical multiplication by a fractional line bundle corresponding to
\[
\sw \in H^2(\sX, \R) \cong \Pic(\sX) \otimes \R.
\]
\end{theorem}
\begin{proof}
Consider the class $\Stab^{\sX, Ell}(p)$. The $q\to 0$ limit of $\Stab^{\sX, Ell}(p)$ shifted by $z \mapsto zq^\sw$ is an integral (with respect to equivariant paramters) K-theoretic class
\[
\Gamma(p) \in \sK_\sT(\sX) \otimes \Q(z).
\]
We know that for generic $\sw'$
\[
\lim_{q\to 0} T_{pr}(zq^{\sw'}, a) = A^{[\sw'], \sX}_{pr} = {\Stab_\sigma}^{[\sw'], \sX, K}(p)|_r
\]
is the K-theoretic stable envelope.
Consider the matrix
\[
Z'' = \left(
\lim_{q\to 0} T_{pr}(zq^\sw, a)
\right)
\left(
\lim_{q\to 0} T_{pr}(zq^{\sw'}, a)
\right)^{-1}, 
\]
By \cite{pcmilect} the correspondence inverse to the stable envelope is the stable envelope with the opposite slope, polarization and equivariant chamber. 
It follows that $Z''$ is an integral correspondence in
\[
\sK_\sT(\sX^\sA \times \sX^\sA) \otimes \Q(z),
\]
and thus its matrix elements are {\it monomials in equivariant parameters}.

By localization, the matrix elements of $Z''$ can be computed as
\[
Z_{pr}'' = \sum_{l \in \sX^{\sA}} \frac{\Stab_{-\sigma}^{[-\sw'], \sX, K}(r)|_l \cdot \Gamma(p)|_l}{\bigwedge^\bullet(T_l \sX^\vee)}
\]
It follows that $Z_{pr}''$ must be monomial in equivariant parameters.
The terms in the numerator satisfy the window conditions that the off-diagonal components of K-theoretic stable envelopes are contained in the diagonal ones after appropriate fractional shift
\[
\deg_\sA \Gamma(p)|_l + \chi_p(-, \sw) \subset \deg_\sA \Stab_\sigma(l)|_l + \chi_l(-, \sw), 
\]
\[
\deg_\sA \Stab_{-\sigma}^{[-\sw'],\sX, K}(r)|_l - \chi_r(-, \sw') \subset \Stab_{-\sigma}^{[-\sw'],\sX, K}(l)|_l -  \chi_l(-, \sw').
\]
For generic $\sw$ it is proved in \cite{AOElliptic}, but the proof uses only quasiperiods of elliptic stable envelopes and is valid for any $\sw$. 
Strictly speaking, the term $\Stab_\sigma(l)|_l$ is not defined since the slope $\sw$ is not generic, but it should be understood as the character prescribed by the normalization condition at the diagonal:
\[
\Stab_{\sigma, T^{1/2}}^{[\sw], \sX, K}(l)|_l = 
\sqrt{\frac{\det T_{l,<0}}
{\det T^{1/2}_l}}
\bigwedge\nolimits^\bullet (T_{l, <0}^\vee),
\]
and it does not depend on the slope $\sw$.
Then, since
\[
T_l = T_{l,<0} + T_{l, >0} = T^{1/2}_l + T^{1/2}_{l, opp},
\]
we have cancellation of the determinants, and
\[
\Stab_\sigma(l)|_l \cdot \Stab_{-\sigma}(l)|_l = \bigwedge\nolimits^\bullet(T_l^\vee).
\]
Then it follows that
\[
Z_{pr}'' \cdot a^{\chi_p(-, \sw) - \chi_l(-, \sw) - \chi_r(-, \sw') + \chi_l(-, \sw')}
\]
is bounded at any infinity of the equivariant torus $\sA$.
Since $Z_{pr}''$ is a monomial in equivariant variables, and $\sw'$ is arbitrary perturbation of $\sw$, it follows that
\[
Z_{pr}'' \cdot a^{\chi_p(-, \sw) - \chi_r(-, \sw)} \in \Q(z, \hbar)
\]
is a matrix independent of equivariant variables.
\end{proof}

Since the matrix coefficients of $Z_{pr}''$ should have integer exponents in equivariant variables, we obtain a corollary:
\begin{corollary}
Let $\sw$ be a slope on exactly one wall. Then
\[
Z_{pr}'' \not = 0 \Rightarrow \chi_p - \chi_r = [C]\otimes v \in H_2(\sX, \Z)_{eff} \otimes \sA^\wedge_{>0}.
\]
\end{corollary}

\section{Identification of the matrix $Z'$}

For ample and anti-ample perturbations of a slope $\sw$, it is possible to describe the $Z'$-matrix in \ref{mainfactorization} explicitly in terms of the symplectic dual variety.

The slope $\sw$ corresponds to a subvariety in $\sX^!$ constructed as follows.
Under the isomorphism $\kappa$, the slope $\sw \in H^2(\sX, \R)$ corresponds to an element
\[\kappa(\sw) \in \Lie_\R(\sA).
\]
It generates some subgroup $\nu_{\kappa(\sw)} \subset \sA$, 
and let
$
Y_\sw = (\sX^!)^{\nu_{\kappa(\sw)}} \subset \sX^!
$
be the $\nu_{\kappa(\sw)}$-fixed subvariety. 
\begin{theorem} \label{mainfactorizationexplicit}
Let $\sw \in \Walls(\sX)$ and $\eps$ be an infinitesimal ample (or anti-ample) slope for $\sX$, such that $\sw' = \sw + \eps$ is a regular slope. Then
$Z'$ is conjugate to the K-theoretic stable envelopes of $Y_\sw$ with infinitesimal ample (respectively, anti-ample) slope:
\[
Z' = H \tilde Z H^{-1},
\]
where
\[
\tilde Z = \frac{\Stab_{\sigma^!}^{\fD_\pm(Y_\sw),K}(p^!)|_{r^!}}{\Stab_{\sigma^!}^{\fD_\pm(Y_\sw),K}(r^!)|_{r^!}},
\]
and $H$ is the same as in \ref{limitnormal} but for $\sX^!$ instead of $\sX$.
\end{theorem}
\begin{proof}
Let us prove the theorem for  $\eps \in \fD_+(\sX)$. The idea is to express $Z'$ as a certain limit of $\tilde T$. By the  factorization theorem \ref{mainfactorization}
\[
\L_{-\sw} \lim_{q\to 0} \tilde T^\sX(zq^\sw, a) \L_{\sw}
=
\L_{-\sw} Z'' A^{[\sw'], \sX} \L_{\sw} = 
Z' \L_{-\sw} A^{[\sw'], \sX} \L_{\sw}.
\]
Since the slope $\eps$ is ample, we have
\[
\chi_p(\eps, \fC) - \chi_r(\eps, \fC) > 0
\]
for $p>r$, and from the K-theoretic window condition, it follows that
\[
\lim_{a \to 0} \L_{-\sw} A^{[\sw'], \sX} \L_{\sw} = 1.
\]
Then the $Z$-matrix for $\sX$ can be recovered as
\[
Z' = \lim_{a \to 0} \L_{-\sw} \lim_{q\to 0} \tilde T^\sX(zq^\sw, a) \L_{\sw}
\]
For the symplectic dual variety $\sX^!$, it becomes the limit we already discussed, which concludes the proof.

\end{proof}

\section{Wall-crossing operators} \label{wall-crossing}

We will now define the wall-crossing operator for a slope $\sw \in H^2(\sX, \R)$. For a generic ample $\eps \in \fD_+(\sX)$ the slopes $\sw \pm \eps$ are regular, and thus define K-theoretic stable envelopes.
The wall R-matrix is defined as the transition matrix between two such bases:
\[
\sR^\sX_\sigma (\sw) = \Stab_\sigma^{[\sw-\eps]} \left(\Stab_\sigma^{[\sw + \eps]}\right)^{-1}.
\]
The wall operators for $\sw$ belonging to exactly one hyperplane play a fundamental role. Other wall operators can be factorized into them.
For example, the distinguished operator is $\sR_\sigma(0)$, describing the transition from ample to anti-ample slopes. It has a factorization
\[
\sR_\sigma(0) = \prod_{\text{walls passing through 0}} \sR^\sX(\sw).
\]
\begin{figure}[h]
    \centering
    \begin{tikzpicture} [scale=2]
    \draw (-1,0) -- (1,0) node [anchor = west] {$\sw_3$} ;
    \draw [rotate=60] (-1,0) -- (1,0) node [anchor = south] {$\sw_1$} ;
    \draw [rotate=120] (-1,0) -- (1,0) node [anchor = south] {$\sw_2$};
    
    \draw [<->, blue, rotate=30] (-.5,0) -- (.5,0) node [anchor=west] {$\sR(0)$} ;
    
    \draw [<->, blue, rotate=30] (-.25,.5) -- (.25,.5) node [anchor=south] {$\sR(\sw_2)$} ;
    
    \draw [<->, blue, rotate=30] (.15,.25) -- (.65,.25) node [anchor=west] {$\sR(\sw_1)$} ;
    
    \draw [<->, blue, rotate=30] (-.65,.25) node [anchor=north east] {$\sR(\sw_3)$} -- (-.15,.25)  ;
    
    \node at (0.7,-0.5) {$\Lie_\R(\sK)$};
    
    \end{tikzpicture}
    \caption{Example of factorization $\sR(0) = \sR(\sw_1) \sR(\sw_2) \sR(\sw_3)$.}
\end{figure}
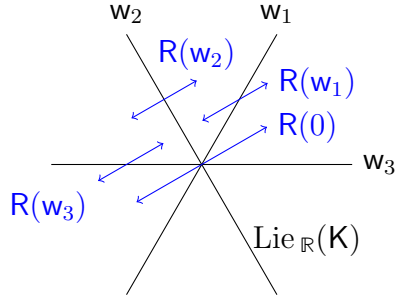

The wall matrices for symplectic dual varieties are related.
\begin{theorem}
Let $\sw \in H^2(\sX, \R)$ and $Y_\sw \subset \sX^!$ be the corresponding subvariety. Then $\sR^\sX_\sigma(\sw)$ and $\sR^{Y_\sw}_{-\sigma^!}(0)$ are conjugate by a diagonal matrix.
\end{theorem}
\begin{proof}
Follows from two factorizations of a limit of the elliptic stable envelopes: to ample and anti-ample slopes.
\end{proof}

\begin{remark}
Since they depend on equivariant parameters for dual varieties (which are K\"ahler for each other), they must depend significantly only on the parameter $\hbar$.
\end{remark}

\begin{corollary}
If $p^!$ and $r^!$ belong to different components of $Y_s$, then
\[
\sR^\sX_\sigma(\sw) = 0.
\]
\end{corollary}

These theorems factor a wall R-matrix $\sR^\sX_\sigma(\sw)$ into a product of wall R-matrices associated to $Y_s$, which can be further factorized using 3d mirror dual to $Y_s$. This leads to a factorization into "elementary" blocks, which cannot be further factorized.

\section{Application: Hilbert scheme}

The Hilbert scheme of points $\sX = \Hilb^n(\C^2)$ is known \cite{OkTalk} to be self-dual:
\[
\sX^! \cong \sX.
\]
The Picard group is generated by the tautological bundle $\cO(1)$. We identify $\Lie_\R(\sK) \cong \R$.
Upon this identification, the walls of $\sX$ are located at the following points
\[
\Walls(\sX) = \left\{
\frac{a}{b} \in \Q: |b| \leq n
\right\}
\]
For a slope $\sw = a/b$, the subvariety $Y_\sw \subset \sX^!$ is a Nakajima quiver variety for the cyclic quiver with $b$ vertices: the framing dimension vector is (1,0,...,0) and we take the union over all dimensions of the vertices.
Its K-theory is known to be the Fock module for the quantum toroidal algebra $\cU_\hbar(\widehat{\widehat{\mathfrak{gl}}}_b)$. The K-theoretic stable envelopes of $\sX$ with infinitesimal ample and anti-ample slopes correspond to the so-called {\it standard} and {\it co-standard} bases of the Fock module. Then our result \ref{mainfactorizationexplicit} implies
\begin{theorem}
The wall $R$-matrix $\sR(\sw)$ for $\sw = \frac{a}{b}$ coincides with the transition matrix from the standatd to the costandard basis in the Fock module for the quantum toroidal algebra $\cU_\hbar(\widehat{\widehat{\mathfrak{gl}}}_b)$
\end{theorem}
This proves the main conjecture of E.Gorsky and A.Negut in \cite{NegGor}.

\chapter{Quantum difference equations and shift operators} \label{qde}

{\it This chapter is less formal than the others and is devoted to outlining main directions of current research.}

The theory discussed in previous chapters sheds light on various fundamental questions of K-theoretic curve counting. The vortex partition function has been studied by physicists and mathematically can be formulated using the quasimap moduli space
\[
\QM(\sX) = \{\text{stable maps } f: \P^1 \to \sX\}/\cong
\]
Note that the domain is parametrized, and we consider quasimaps equivariantly with respect to automorphisms of the domain and of the target. 

Quasimap moduli spaces are defined for the target being a GIT quotient, but they have very simple meaning in case of Nakajima quiver varieties. In this case the framings and the vector spaces associated with the vertices are replaced by vector bundles over $\P^1$, and arrows are replaced by morphisms of bundles.

The moduli space of quasimaps has a perfect obstruction theory with the tangent space
\[
T_\vir = H^\bullet\left(
\mathscr{R} - (1+\hbar^{-1}) \End \V
\right),
\]
where $\mathscr{R}$ is the bundle associated with representations of a quiver (including the cotangent part), and $\V$ is the bundle associated with vertices. The two terms we are subtracting are related to the division by the gauge group and the moment map equations.

The symmetrized structure sheaf is defined as
\[
\widehat{\cO}_\vir = \cO_\vir \otimes \sqrt{
\cK_\vir \frac{\det \cT^{1/2}|_\infty}{\det \cT^{1/2}|_0}
}
\]
The presence of the virtual canonical sheaf $\cK_\vir$ results in the appearance of $\ahat(w)$ instead of $(1-w^{-1})$ for the tangent weights $w$ in the localization formula, and the twist by the determinants of the polarization is needed to get rid of square roots.

We consider the 1-dimensional torus $\C^\times_q$ acting on $\P^1$, so that the tangent spaces at $0$ and $\infty$ are $q$ and $q^{-1}$ respectively. This allows us to define K-theoretic integration over non-proper spaces by localization. As was observed by M.Aganagic and A.Okounkov, this $q$ is the same $q$ appearing in theta-functions for elliptic stable envelopes, which is a fundamental geometric fact.

For a more fundamental introduction to K-theoretic computations in enumerative geometry see \cite{pcmilect}.

\section{Toy example}

Let us start with an example where everything can be made as explicit as possible and understood from different prospectives.

Consider the moduli space of quasimaps to $\sX = T^* \P^0 = T^* \C / / \C^\times$.
A quasimap to $\sX$ of degree $d$ is a line bundle $\cO(d)$ over $\P^1$ together with two maps
\[
x: \cO \to \cO(d), \ \ \ y: \cO(d) \to \cO
\]
satisfying the moment map equation
\[
\mu = xy = 0.
\]
It follows that either $x$ or $y$ should be 0, and that is why we have two stability conditions: $x\neq 0$ and $y\neq 0$.
They correspond to two signs of characters of the gauge group $\C^\times$ and will give expansions of the vertex functions at $z=0$ and $z=\infty$.

First, consider the case $x \neq 0$. It follows that the degree of the bundle should be nonnegative: $d \geq 0$.
The off-shell tangent space is
\[
T(x) = x + \hbar^{-1} x^{-1} - (1+\hbar^{-1}),
\]
and the unique fixed point for $\sX$ corresponds to $x=1$, for which $T=0$, as expected.

A unique fixed point for quasimaps of degree $d$ corresponds to the bundle $\cO(d)$ with a linearization so that the fibers over $0$ and $\infty$ have weights $q^n$ and $1$, respectively. Nonsingularity of the quasimap at $\infty$ simply means  that the weight over $\infty$ does not depend on $q$.
The tangent space to quasimaps, thus, is
\[
\frac{T(q^n)}{1-q^{-1}} + \frac{T(1)}{1-q} = \sum_{k=1}^n \left(
q^k - \frac{1}{\hbar q^{k-1}}
\right).
\]

By localization, the vertex function is
\[
V_0(z) = \sum_{d\geq 0} \left(
-\frac{z}{\sqrt{h}}
\right)^d \frac{(\hbar)_d}{(q)_d},
\]
where
\[
(x)_d = (1-x)(1-qx)...(1-q^{d-1}x).
\]
Using the $q$-binomial formula, this function can be represented as an infinite product
\[
V_0(z) = S^\bullet \left(
\frac{1-\hbar}{1-q} \frac{z}{\sqrt{\hbar}}
\right)
 = \prod_{i\geq 0}
\frac{1-q^i z/\sqrt{\hbar}}{1-q^i z\sqrt{\hbar}}.
\]
The function $V_0(z)$ satisfies the $q$-difference equation
\[
V_0(qz) = \frac{1-z/\sqrt{\hbar}}{1-z\sqrt{\hbar}} V_0(z),
\]
so the quantum difference connection is
\[
M^\sX(a,z) = \frac{1-z/\sqrt{\hbar}}{1-z\sqrt{\hbar}}.
\]

Similarly, for the other stability condition we have
bundles $\cO(-d)$ with fibers over $0$ and $\infty$ having weights $\hbar^{-1}q^{-n}$ and $\hbar^{-1}$. By localization, we obtain the vertex function for this stability condition:
\[
V_\infty(z) = \sum_{d\leq 0} \left(-\frac{z\sqrt{\hbar}}{q}
\right)^d  \frac{(\hbar)_d}{(q)_d}.
\]
which is a trivialization of the same difference connection near $z=\infty$.

Its monodromy is
\[
\Mon = V_0(z) V_\infty(z)^{-1} \sim \frac{\vartheta(z/\sqrt{\hbar})}{\vartheta(z\sqrt{\hbar})}
\]

The monodromy is equal to the elliptic R-matrix for the variety $\sX^!$ which is a 2-dimensional vector space with character $\frac{z}{\sqrt\hbar} + \frac{1}{z\sqrt\hbar}$. Here $z$ is an equivariant parameter, and $\hbar$ is the weight of the symplectic form.

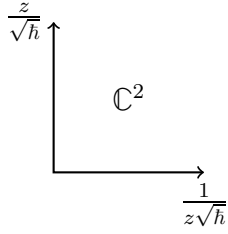
\begin{figure}
\begin{center}
\begin{tikzpicture} [scale=2] 
\draw [<->, thick] (1,0) node[anchor=north] {$\frac{1}{z\sqrt\hbar}$} -- (0,0) -- (0,1) node[anchor=east] {$\frac{z}{\sqrt\hbar}$};
\draw (0.5,0.5) node {$\C^2$} ;
\end{tikzpicture}
\end{center}
\caption{Symplectic dual variety $\sX^! \cong \C^2$ for the variety $\sX = T^* \P^0$}.
\end{figure}

The quantum difference equation can be recovered from the monodromy by taking the limit $q\to 0$:
\[
\lim_{q\to 0} \frac{\vartheta(z/\sqrt\hbar)}{\vartheta(z\sqrt\hbar)} = \hbar \frac{1-z/\sqrt{\hbar}}{1-z\sqrt\hbar},
\]
which coincides with $M^\sX(a,z)$ up to a factor $\hbar$.

The variety $\sX^!$ is contractible, so there are no nontrivial curve counts in nonzero degree. In order to compute the shift operator for $\sX^!$ we should consider the bundle
\[
\cO(1) \oplus \cO(-1) \to \P^1.
\]
The tangent space to quasimaps is given by the cohomology:
\[
H^\bullet (\cO(1) \oplus \cO(-1)) = \frac{z}{\sqrt\hbar} + \frac{zq}{\sqrt\hbar}.
\]

Let the polarization be given by the $z/\sqrt\hbar$ direction.
Then, by localization the contribution of the unique fixed point is
\[
\frac{1}{\ahat(z/\sqrt\hbar) \ahat(zq/\sqrt\hbar)} \cdot \sqrt{1/q}.
\]
Similarly, for the gluing matrix we obtain
\[
\frac{1}{\ahat(z/\sqrt\hbar) \ahat(zq/\sqrt\hbar)}
\]
Dividing one by the other, we get
\[
S = \frac{1-z\sqrt\hbar}{1-zq/\sqrt\hbar},
\]
which is equivalent to $M^\sX(z)$ after certain change of variables and multiplication by a monomial.

\section{Elliptic R-matrix and monodromy}

Let us call the variety with another choice of stability condition by $\sX_{\flop}$. The elliptic stable envelopes for $\sX_{\flop}$ can be obtained by restrictions of the same off-shell stable envelopes to different fixed point set, and they have a different normalization. We will consider examples in chapter \ref{explicitexamples}.
\begin{proposition}
Consider the R-matrix related to the change of the equivariant chamber
\[
R_{\fC \to -\fC}^\sX = \Stab_{-\fC}^\sX \Stab_{\fC}^{\sX-1}
\]
 and to the flop
 \[
R_\fC^{\sX \to \sX_{\flop}} = \Stab_\fC^{\sX_{\flop}} \Stab_\fC^{\sX-1}
\]
The corresponding elliptic correspondences in $\sX^\sA\times \sX^\sA$ are balanced in equivariant and in K\"ahler parameters (separately).
\end{proposition}
\begin{proof}
Exercise to the reader.
\end{proof}

\section{Shift of the K\"ahler parameters and monodromy}

We are interested in monodromies of difference equations. One could study differential equations and obtain K-theoretic wall operators as their monodromies, see \cite{smirnov2021quantum}.

Consider a quantum difference connection given by a rational operator $M(z)$:
\[
\Psi(zq) = M(z) \Psi(z).
\]
For simplicity, let us assume that $z$ is a single variable.
It is natural to define its solution as a gauge transformation of the equation, which can be done in a neighborhood of any point, for example, 0 and $\infty$:
\[
\Psi_0(zq) M(0) = M(z) \Psi_0(z),
\]
\[
\Psi_\infty(zq) M(\infty) = M(z) \Psi_\infty(z).
\]
The monodromy is defined to be a transformation
\[
\Mon(z) = \Psi_0^{-1}(z) \Psi_\infty(z).
\]
It satisfies the difference equation
\[
\Mon(zq) = M(0) \Mon(z) M(\infty)^{-1},
\]
and thus is a matrix of elliptic functions.
In general, for any rational $M(z)$ it is very hard to compute the monodromy explicitly. However, for the equations coming from enumerative geometry, their monodromy has a description in geometric terms.

Let us consider monodromies associated with the K-theoretic quantum difference equation.
The monodromy of the quantum difference equation is given by the elliptic R-matrix:
\[
\Mon = \det T^{1/2} \cdot \Stab^{\sX_{\flop}-1} \cdot \Stab^{\sX} \cdot \left(\det T^{1/2} \right)^{-1}
\]
Unlike the case of the shift operator, now $\Mon$ is an operator on $K_\sT(\sX)$.

The monodromy of a difference operator $M(z)$ can be understood as a regularized product
\[
\Mon \sim \prod_{i \in \Z}^\rightarrow M(zq^i)
\]
Each factor has its own factorization \cite{OS}
\[
M(zq^i) \sim \prod_{i\leq \sw < i+1} \B_\sw
\]
with respect to the poles of the K\"ahler variables, as explained in the appendix.
The operators $\B_\sw$ are also called the wall-crossing operators, but they should not be confused with the operators in the section \ref{wall-crossing}, since they are not triangular. It turns out that it is possible to reconstruct individual  operators $\B_\sw$ from the monodromy.

The following conjecture was proposed by A.Okounkov \cite{OkRodeIsland}:
\begin{conjecture}Denote the following limit
\[
\Mon_\sw = \lim_{q\to 0} \Mon(a, zq^\sw).
\]
Then the wall operator has the following description
\[
\B_\sw(z_\text{shifted}) = \Mon_\sw \cdot \Mon_{\sw + \eps}^{-1}
\]
\end{conjecture}
The next conjecture follows from the previous one using \ref{mainfactorizationexplicit}:
\begin{conjecture}
In the stable basis the operator $\B_\sw$ is conjugate to the R-matrix for $\sX^{!, \nu_\sw} \subset \sX$.
\end{conjecture}

The quantum difference equation can be viewed as a generalization of the quantum Knizhnik-Zamolodchikov equation, since the operator of the equation is a product of $R$-matrices corresponding to different varieties.

In the case when the K\"ahler torus is higher-dimensional, the situation is the same.  The solutions are expanded as series in K\"ahler variables whose exponents belong to certain cone, determined by a stability condition. Monodromy is the transition between different stability conditions. It was proved in \cite{AOElliptic} that such operators are equal to the elliptic dynamical R-matrix for $\sX^!$. The chambers for $\sX^!$ correspond to stability conditions for $\sX$ we are dealing with.

For an explicit example of such computation, see    chapter \ref{explicitexamples}.

\section{Minuscule shift operators and qKZ}

There are operators shifting equivariant variables. They commute with the quantum difference equations, and are determined uniquely by this property together with the normalization condition.  From some point of view they are easier to describe geometrically, especially when the symplectic dual variety is not known.

The simplest case is the minuscule shift operator which corresponds to a shift by a minuscule cocharacter.
A cocharacter $\sw$ is called minuscule if the algebra $\C[\sX_0]$ is generated by elements of $\sw$-degree $0, \pm 1$. For example, shifts of equivariant variables corresponding to the {\it framing torus} are generated by minuscule shifts.

It was shown \cite{pcmilect} that the minuscule shift operator  has a simple description in the basis of stable envelopes:
\[
\langle \Stab_{-, T^{1/2}_{opp}, -\L} | S_\sw  | \Stab_{-, T^{1/2}, \L} \rangle = \pm z^{\deg}.
\]
In other words, the shift operator is just the R-matrix times the classical multiplication by a line bundle for $\sX^!$. The corresponding difference equation can be identified with the quantum Knizhnik-Zamolodchikov equation for certain quantum groups, introduced in \cite{FR} based on pure representation-theoretic considerations. This result was proved in \cite{pcmilect} by a rigidity argument, given that in the stable basis such operator is holomorphic in equivariant variables.

For more general shifts, there is a similar factorization to the wall-crossing operators with respect to the poles in equivariant variables.

\section{Shift by a general cocharacter}

Consider a shift by arbitrary integral cocharacter $\sw$. The poles of the shift operators in equivariant parameters are very easy to analyze.

It is known that $\C[\sX_0]$ is a finitely generated algebra, so we have a surjective homomorphism
\[
\C[\xi_1, ..., \xi_n] \to \C[\sX_0] \to 0,
\]
which corresponds to embedding $\sX_0$ into a vector space
$\sX_0 \to V$.

We can always achieve that the generators $\xi_i$ are homogeneous and have some weights $\lambda_i$. Thus, we have a proper map
\[
\QM(\sX) \to \QM(V).
\]
Quasimaps to $V$ are very easy to describe: we have a section of $\cO(\inner{\lambda_i}{\sw})$ for each coordinate $\xi_i$. If the degree of the line is greater then 1, there are sections going to infinity with fixed values at $0$, $\infty$. The corresponding contribution is given by the cohomology, and computing the character we obtain
\begin{theorem}
The operator $S_\sw$ in the stable basis has no poles outside of
\[
a^{\lambda_i} q^i, \ \ i=1, ... , \inner{\lambda_i}{\sw} - 1 
\]
with respect to the equivariant parameter.
\end{theorem}

\begin{example}
Consider the example of the Hilbert scheme of points in $\C^2$:
\[
\sX = \Hilb(\C^2, n).
\]
From  geometric invariant theory, the ring of invariant functions on the prequotient is generated by
\[
\tr(X_1^i X_2^j) \ \ \text{for} \ \ i+j\leq n.
\]
The bundle corresponding to $\tr(X_1^i X_2^j)$ is $\cO(i-j)$.
Thus, the poles of the shift operator are
\[
\{t_1^i t_2^j q^k\}_{k=1}^{i-j-1}.
\]
Let us consider small values of $n$:
\begin{enumerate}
\item For $n=1$ the character is minuscule, and the shift operator is integral;
\item For $n=2$ the only pole is at $t_1^2 q$;
\item For $n=3$ the only poles are at
\[
t_1^2 q, t_1^3 q, t_1^3 q^2;
\]
\item For $n=4$ the only poles are at
\[
t_1^2 q, t_1^3 q, t_1^3 q^2, t_1^4 q, t_1^4 q^2, t_1^4 q^3, t_1^3 t_2 q,
\]
and so on.
\end{enumerate}

\end{example}

For simplicity, assume that the equivariant torus is 1-dimensional. By the Wiener-Hopf factorization theorem (see the appendix),  there is a factorization of the shift operator corresponding to the poles with respect to the equivariant parameter:
\[
S_\sw = \prod S_{\sw_i}.
\]

Define
\[
\Mon(a) = \Stab^{\sX}_{\fC} \cdot \Stab^{\sX - 1}_{-\fC}.
\]
It is an operator on $K(\sX^\sA)$, and let
\[
\Mon_\sw = \lim_{q \to 0} \Mon(aq^\sw),
\]
where the limit is taken in the natural basis in $K(\sX^\sA)$.

The following conjectures were proposed by A.Okounkov \cite{OkRodeIsland}.
\begin{conjecture}
The operator $S_\sw$ in the stable basis is conjugate to
\[
\Mon_\sw \cdot \Mon_{\sw + \eps}^{-1}.
\]
\end{conjecture}
The next conjecture follows from the previous one using \ref{mainfactorizationexplicit}:
\begin{conjecture}
In the stable basis the operator $S_\sw$ is conjugate to the R-matrix for $\sX^{\nu_\sw} \subset \sX$.
\end{conjecture}

The shift operator also takes the form of generalized qKZ equation.

\chapter{Explicit examples} \label{explicitexamples}

In this section we discuss how to perform computations explicitly for the simplest Nakajima varieties. 

\section{Nakajima quiver varieties}

One of the most important class of symplectic resolutions of singularities are Nakajima quiver varieties. 
There are many excellent papers and books \cite{Nak1, GinzburgLectures} about this construction, and here we just give a brief down-to-earth overview.

Nakajima quiver varieties are encoded by graphs.

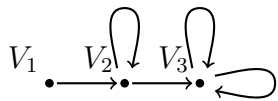
\begin{figure}[h]
    \centering
    \begin{tikzpicture}
    \draw[fill=black] (0,0) circle (0.05) node[anchor=south east] {$V_1$};
    \draw[fill=black] (1,0) circle (0.05) node[anchor=south east] {$V_2$};
    \draw[fill=black] (2,0) circle (0.05) node[anchor=south east] {$V_3$};
    \draw [thick, ->] (0.9,0.2) to[out=100, in=180] (1,1) to[out=0, in=80] (1.1,0.2);
    \draw [thick, ->] (1.9,0.2) to[out=100, in=180] (2,1) to[out=0, in=80] (2.1,0.2);
    \draw [thick, ->, rotate around={270: (2,0)}] (1.9,0.2) to[out=100, in=180] (2,1) to[out=0, in=80] (2.1,0.2);
    \draw [thick, ->] (0.1,0) -- (0.9,0);
    \draw [thick, ->] (1.1,0) -- (1.9,0);
    \end{tikzpicture}
    \caption{Nakajima quiver variety $T^*(\bigoplus \Hom(W_i, V_i) \oplus \Hom(V_1, V_2) \oplus \Hom(V_2, V_3) \oplus \End(V_2) \oplus \End(V_3)^{\oplus 2})//\prod GL(V_i)$}
    \label{Nakajima}
\end{figure}

For each vertex $V_i$ we introduce the framing space $W_i$.
The Nakajima quiver variety is defined as
\[
\sX = T^* Rep //\prod GL(V_i).,
\]
where the representation space of the graph is
\[
Rep = \bigoplus_i \Hom(W_i, V_i) \oplus \bigoplus_{i\to j} \Hom(V_i, V_j).
\]
In other words, they are Hamiltonian reductions for products of general linear groups of cotangent bundles to representations which are direct sums of defining representations and hom's between defining representations.
They also depend on the stability parameter which is a character
\[
\prod GL(V_i) \to \C^\times.
\]

Some of the simplest and most fundamental examples of Nakajima quiver varieties include Grassmannians and the Hilbert scheme of points in $\C^2$.

\section{Grassmannian}

Stable envelopes for the Grassmannian $T^*Gr(k,n)$ were constructed in \cite{AOElliptic} using abelianization.
The Grassmannian is a Nakajima quiver variety with a single vertex and no loops, as indicated in Figure \ref{grassmannian_quiver}
\begin{figure}[h]
    \centering
    \begin{tikzpicture}
    \draw [thick] (-0.1,-0.1)--(-0.1,0.1)--(0.1,0.1)--(0.1,-0.1)--cycle;
    \draw [thick] (1,0) circle (0.1);
    \draw [->, thick] (0.2,0.1) -- (0.8,0.1);
    \draw [<-, thick] (0.2,-0.1) -- (0.8,-0.1);
    \node at (-0.5,0) {$\C^n$};
    \node at (1.5,0) {$\C^k$};
    \node at (0.5,0.3) {$A$};
    \node at (0.5,-0.3) {$B$};
    \end{tikzpicture}
    \caption{Grassmannian: $\{A,B: AB=0\}//GL(n)$.}
    \label{grassmannian_quiver}
\end{figure}

Then we have
\[
X = T^* Gr(k,n) =  T^*\Hom(W,V)//GL(V).
\]
Let the torus
\[
A = \left\{
\begin{pmatrix}
a_1 & & \\
& \ddots & \\
 & & a_n
\end{pmatrix}
\right\}
\]
act on $W$ with the character
\[
W = a_1 + ... + a_n,
\]
and $G = \C^\times_{x_1} \times ... \times \C^\times_{x_n}$ act on $V = \C$ with the character $x_1 + ... + x_n$. 

Let us choose the polarization given by the base directions
\[
T^{1/2} = \overline{W} V - V\overline V.
\]
The character of the tangent bundle is
\[
T = \overline W V + \hbar^{-1} W \overline V - (1 + \hbar^{-1}) V \overline V = 
\sum \left(\frac{x_i}{a_j} + \hbar^{-1} \frac{a_j}{x_i}
\right) - (1+\hbar^{-1}) \sum \frac{x_i}{x_j}.
\]
The sum corresponds to the representation of the quiver, and the subtraction corresponds to factorization and the moment map equations.

Fixed points in $\sX$ are the coordinate $k$-dimensional subspaces in $W$ and are parametrized by $k$-subsets $\alpha \{1, ..., n\}$

Assume that the positive chamber is
\[
a_i = \xi^i, \xi \to \infty.
\]

Then the tangent space at $x_i = a_{\alpha_i}$ has the repelling part
\[
T_{<0} =
\sum_{\alpha_i<j} \frac{x_{i}}{a_j} + \sum_{\alpha_i>j} \hbar^{-1} \frac{a_j}{x_{i}}
 - (1+\hbar^{-1}) \sum_{\alpha_i < \alpha_j} \frac{x_{i}}{x_{j}}
\]
The $\sA$-fixed part of the polartization is
\[
T^{1/2}_{=0} = \sum_i \frac{x_{i}}{a_{i}} - k.
\]

The formula for the off-shell elliptic stable envelopes in \cite{AOElliptic} gives
\[
\Stab(\alpha) = \vartheta(T_{<0}) \cdot \left(
\text{$z,\hbar$-refinement of $T^{1/2}_{=0}$}
\right).
\]
Let us call the first factor $\Sh = \vartheta(T_{<0})$ as Shenfeld factor, and
the second factor $\Rt$, which needs refinement, as the "root" factor (it was motivated by the examples f the Hilbert scheme, considered later).

The refinement by K\"ahler parameters and $\hbar$ is uniquely fixed from the quasiperiods.

\begin{example}
For example, consider the case $\sX = T^* Gr(2,4)$ and the point $\alpha = \{1,3\} \subset \{1,2,3,4\}$. At the fixed point, we have
\[
x_1 = a_1, \ \ x_2 = a_3.
\]
Then
\[
\Sh = \vartheta(T_{<0}) = \frac{\vartheta(x_1/a_2)\vartheta(x_1/a_3)\vartheta(x_1/a_4)\vartheta(x_2/a_4)\vartheta(a_1/hx_2)\vartheta(a_2/hx_2)}{\vartheta(x_1/x_2)\vartheta(x_1/hx_2)}
\]
The "root" factor must have the form
\[
\Rt = \frac{\vartheta(x_1/a_1 \cdot z^{s_1} \hbar^{s_2})}{\vartheta( z^{s_1} \hbar^{s_2})} \cdot \frac{\vartheta(x_2/a_2 \cdot z^{s_3} \hbar^{s_4})}{\vartheta( z^{s_3} \hbar^{s_4})}, \ \ s_1, s_2, s_3, s_4 \in \Z.
\]
The unknown numbers $s_1, s_2, s_3, s_4$ are uniquely fixed from the quasiperiods. Shift $z \to zq$ should produce the factor of the determinant of the tautological bundle $(x_1 x_2)^{-1}$, thus
\[
s_1 = -1, \ s_3 = -1,
\]
and the shifts $x_i \to x_i q$ should not produce $\hbar$ as automorphy factors, thus 
\[
s_2 = -1, s_4 = 0.
\]

This numbers $s_1, s_2, s_3, s_4$ are the only exponents for which the class $\Sh\cdot \Rt$ has the automorphy factors that are invariant under permutations $x_1 \leftrightarrow x_2$.
\end{example}

\subsection{Factorization}

Consider the case $\sX = T^* \P^2$.
The normalized matrix of stable envelopes is
\[
\tilde T =   \left[ \begin {array}{ccc} 1&{\frac {\vartheta ( {{\hbar}}^{-1}
 ) }{\vartheta ( z{{\hbar}}^{2} ) } \frac{\vartheta ( {
 {z{{\hbar}}^{2}a_{{1}}}/{a_{{2}}}} )  }{ \vartheta
 \left( { {a_{{2}}}/{a_{{1}}}} \right)   }}&{\frac {
\vartheta ( {{\hbar}} ) }{\vartheta \left( z{{\hbar}}^
{2} \right) } \frac{\vartheta \left( { {{\hbar}\,a_{{3}}/a_2}}
 \right) \vartheta \left( { {z{{\hbar}}^{2}a_{{1}}}/{a_{{3}}}}
 \right)}  { \vartheta ( { {a_{{3}}}/{a_{{1}}}} ) 
    \vartheta ( { {a_{{3}}}/{a_{{2}}}}
 )   }}\\ \noalign{\medskip}0&1&\frac {\vartheta
 ( {{\hbar}}^{-1} ) }{\vartheta \left( {\hbar}\,z
 \right) } \frac{\vartheta \left( { {{\hbar}\,za_{{2}}}/{a_{{3}}}}
 \right) }{  \vartheta \left( { {a_{{3}}}/{a_{{2}}}} \right) 
 }\\ \noalign{\medskip}0&0&1\end {array} \right] 
\]

The limit to the wall $\sw = 0 \in \Lie_\R(\sK)$ has the factorization
\begin{multline}
\lim_{q \to 0} \tilde T(a,z) =    \left[ \begin {array}{ccc} 1&{\frac { \left( {\hbar}-1 \right) 
 \left( z{{\hbar}}^{2}a_{{1}}-a_{{2}} \right) }{\sqrt {{\hbar}}
 \left( z{{\hbar}}^{2}-1 \right)  \left( a_{{1}}-a_{{2}} \right) }}
&{\frac { \left( {\hbar}\,a_{{3}}-a_{{2}} \right)  \left( {\hbar
}-1 \right)  \left( z{{\hbar}}^{2}a_{{1}}-a_{{3}} \right) }{{\hbar}\, \left( z{{\hbar}}^{2}-1 \right)  \left( a_{{1}}-a_{{3}}
 \right)  \left( a_{{2}}-a_{{3}} \right) }}\\ \noalign{\medskip}0&1&{
\frac { \left( {\hbar}-1 \right)  \left( {\hbar}\,za_{{2}}-a_{{3
}} \right) }{\sqrt {{\hbar}} \left( {\hbar}\,z-1 \right) 
 \left( a_{{2}}-a_{{3}} \right) }}\\ \noalign{\medskip}0&0&1
\end {array} \right]  = \\
=
 \left[ \begin {array}{ccc} 1&{\frac {{{\hbar}}^{3/2} \left( {\hbar}-1 \right) z}{z{{\hbar}}^{2}-1}}&-{\frac {{\hbar}\, \left( 
{\hbar}-1 \right) z}{z{{\hbar}}^{2}-1}}\\ \noalign{\medskip}0&1&
{\frac {\sqrt {{\hbar}} \left( {\hbar}-1 \right) z}{{\hbar}\,
z-1}}\\ \noalign{\medskip}0&0&1\end {array} \right] 
\cdot
 \left[ \begin {array}{ccc} 1&{\frac { \left( {\hbar}-1 \right) a_{
{2}}}{\sqrt {{\hbar}} \left( a_{{1}}-a_{{2}} \right) }}&{\frac {a_{
{3}} \left( {\hbar}\,a_{{3}}-a_{{2}} \right)  \left( {\hbar}-1
 \right) }{{\hbar}\, \left( a_{{1}}-a_{{3}} \right)  \left( a_{{2}}
-a_{{3}} \right) }}\\ \noalign{\medskip}0&1&{\frac { \left( {\hbar}
-1 \right) a_{{3}}}{\sqrt {{\hbar}} \left( a_{{2}}-a_{{3}} \right) 
}}\\ \noalign{\medskip}0&0&1\end {array} \right] = \\
=
 \left[ \begin {array}{ccc} 1&{\frac {{\hbar}-1}{\sqrt {{\hbar}}
 \left( z{{\hbar}}^{2}-1 \right) }}&{\frac {1-{\hbar}}{z{{\hbar}}^{2}-1}}\\ \noalign{\medskip}0&1&{\frac {{\hbar}-1}{\sqrt {{
\hbar}} \left( {\hbar}\,z-1 \right) }}\\ \noalign{\medskip}0&0&1
\end {array} \right] 
 \cdot
  \left[ \begin {array}{ccc} 1&{\frac { \left( {\hbar}-1 \right) a_{
{1}}}{\sqrt {{\hbar}} \left( a_{{1}}-a_{{2}} \right) }}&{\frac {
 \left( {\hbar}\,a_{{3}}-a_{{2}} \right)  \left( {\hbar}-1
 \right) a_{{1}}}{{\hbar}\, \left( a_{{1}}-a_{{3}} \right)  \left( 
a_{{2}}-a_{{3}} \right) }}\\ \noalign{\medskip}0&1&{\frac { \left( {
\hbar}-1 \right) a_{{2}}}{\sqrt {{\hbar}} \left( a_{{2}}-a_{{3}}
 \right) }}\\ \noalign{\medskip}0&0&1\end {array} \right] 
\end{multline}
The first factorization involves stable envelopes for the ample slope for $\sX^!$ and stable envelopes for the slope $+\eps$ for $\sX$. The other involves the anti-ample slope for $\sX^!$ and the slope $-\eps$ for $\sX$. It implies that the wall R-matrix for $\sX$ is related to the transition matrix between ample and anti-ample slopes for $\sX^!$.

\subsection{Difference equations}

One of the most important problems is the computation of the gluing matrix.
The K-theory of $\bigsqcup_{k=0}^n T^* Gr(k,n)$ is the representation $\otimes_i \C^2(a_i)$ of the group  $\cU_\hbar(sl_2)$. The K-theory of $T^* \P^{n-1}$ is the subspace of the weight one below the highest weight. It follows that the gluing operator for $T^* \P^{n-1}$ can be obtained from the part of degree 1 as
\[
\Glue(z) = 1 + \frac{z}{1-(-\sqrt\hbar)^n \cdot z} \Glue_\text{degree 1}.
\]

The degree-1 part of the gluing matrix can be computed directly:
\[
\Glue_\text{degree $1$} =  \left[ \begin {array}{cc} {\frac { \left( {\hbar}\,a_{{1}}-a_{{2}}
 \right)  \left( {\hbar}-1 \right) }{ \left( a_{{1}}-a_{{2}}
 \right) {\hbar}}}&{\frac {a_{{2}} \left( {\hbar}\,a_{{2}}-a_{{1
}} \right)  \left( {\hbar}-1 \right) }{{\hbar}\,a_{{1}} \left( a
_{{1}}-a_{{2}} \right) }}\\ \noalign{\medskip}-{\frac {a_{{1}} \left( 
{\hbar}\,a_{{1}}-a_{{2}} \right)  \left( {\hbar}-1 \right) }{
 \left( a_{{1}}-a_{{2}} \right) a_{{2}}{\hbar}}}&-{\frac { \left( {
\hbar}\,a_{{2}}-a_{{1}} \right)  \left( {\hbar}-1 \right) }{
 \left( a_{{1}}-a_{{2}} \right) {\hbar}}}\end {array} \right] 
\]

For this variety the operator of the quantum difference equation can be obtained as
\[
M(z) = \Glue(zq) \cO(1),
\]
and, explicitly, 
\[
M(z) =   \left[ \begin {array}{cc} -{\frac {a_{{1}} \left( {{\hbar}}^{2}qza
_{{2}}+ \left( -qza_{{1}}-zqa_{{2}}+a_{{1}}-a_{{2}} \right) {\hbar}
+zqa_{{2}} \right) }{{\hbar}\, \left( {\hbar}\,qz-1 \right) 
 \left( a_{{1}}-a_{{2}} \right) }}&-{\frac {za_{{2}}^{2}q \left( {
\hbar}\,a_{{2}}-a_{{1}} \right)  \left( {\hbar}-1 \right) }{{
\hbar}\, \left( {\hbar}\,qz-1 \right)  \left( a_{{1}}-a_{{2}}
 \right) a_{{1}}}}\\ \noalign{\medskip}{\frac {za_{{1}}^{2}q \left( 
{\hbar}-1 \right)  \left( {\hbar}\,a_{{1}}-a_{{2}} \right) }{{
\hbar}\, \left( {\hbar}\,qz-1 \right) a_{{2}} \left( a_{{1}}-a_{
{2}} \right) }}&{\frac {a_{{2}} \left( {{\hbar}}^{2}qza_{{1}}+
 \left( -qza_{{1}}-zqa_{{2}}-a_{{1}}+a_{{2}} \right) {\hbar}+qza_{{
1}} \right) }{{\hbar}\, \left( {\hbar}\,qz-1 \right)  \left( a_{
{1}}-a_{{2}} \right) }}\end {array} \right] 
\]

Let $S(z)$ be the shift operator corresponding to the cocharacter $a_2 \to a_2 q$. It is defined uniquely up to a scalar by the commutativity with the quantum difference equation:
\[
\left.M(z)\right|_{a_2 \to a_2 q} S(z) = S(zq) M(z),
\]
and we obtain
\[
S(z) =   \left[ \begin {array}{cc} -{\frac { \left( a_{{1}}-a_{{2}} \right) 
\sqrt {{\hbar}}q}{{\hbar}\,a_{{1}}-qa_{{2}}}}+{\frac { \left( {
\hbar}-1 \right) ^{2}a_{{2}}a_{{1}}{q}^{2}z}{\sqrt {{\hbar}}
 \left( qa_{{2}}-a_{{1}} \right)  \left( {\hbar}\,a_{{1}}-qa_{{2}}
 \right) }}&{\frac {a_{{2}}^{2}{q}^{2} \left( {\hbar}-1 \right) 
 \left( {\hbar}\,a_{{2}}-a_{{1}} \right) z}{\sqrt {{\hbar}}
 \left( {\hbar}\,a_{{1}}-qa_{{2}} \right) a_{{1}} \left( qa_{{2}}-a
_{{1}} \right) }}\\ \noalign{\medskip}-{\frac {a_{{1}}^{2} \left( {
\hbar}-1 \right) z}{ \left( qa_{{2}}-a_{{1}} \right) a_{{2}}\sqrt {
{\hbar}}}}&-{\frac { \left( {\hbar}\,a_{{2}}-a_{{1}} \right) z}{
\sqrt {{\hbar}} \left( qa_{{2}}-a_{{1}} \right) }}\end {array}
 \right] 
\]

Let us now compute the same operators from the monodromy.

The matrix of elliptic stable envelopes in the basis of the fixed points is given by
\[
\Stab =  \left[ \begin {array}{cc} \vartheta ( { {a_{{2}}}/({{\hbar}\,
a_{{1}}}}) ) &{\frac {\vartheta \left( {{\hbar}}^{-1} \right) \vartheta ( { {z{\hbar}\,
a_{{1}}}/{a_{{2}}}} )  }{
\vartheta \left( z{\hbar} \right) }}\\ \noalign{\medskip}0&\vartheta ( {
 {a_{{2}}}/{a_{{1}}}} ) \end {array} \right] 
\]
and for the flop
\[
 \Stab_{\flop} = \left[ \begin {array}{cc} {\frac {\vartheta \left( z{{\hbar}}^{2}
 \right) \vartheta ( { {a_
{{2}}}/{a_{{1}}}} ) }{\vartheta \left( z{\hbar} \right) } }&0\\ \noalign{\medskip}{\frac {\vartheta
 \left( {{\hbar}}^{-1} \right) \vartheta
 ( { {za_{{2}}{\hbar}}/{a_{{1}}}} ) }{\vartheta \left( z \right) } }&{\frac {\vartheta
 \left( z{\hbar} \right) \vartheta ( {
 {a_{{2}}}/({{\hbar}\,a_{{1}}}}) ) }{\vartheta \left( z \right) } }\end {array} \right] 
\]

The monodromy in $z$ is
\[
\Mon(z) = \det T^{1/2} \cdot \Stab^{-1} \cdot  \Stab_{\flop} \cdot \left(\det T^{1/2}\right)^{-1}.
\]

We now compute its limit to a wall $\sw = 0 \in H^2(\sX, \R)$:
\[
\Mon_0 =  \left[ \begin {array}{cc} -{\frac { \left( {\hbar}\,a_{{2}}-a_{{1}
} \right)  \left( z{\hbar}-1 \right) }{ \left( a_{{1}}-a_{{2}}
 \right) {\hbar}\, \left( -1+z \right) }}&-{\frac { \left( {\hbar}-1 \right)  \left( z{\hbar}\,a_{{1}}-a_{{2}} \right) }{
 \left( a_{{1}}-a_{{2}} \right) {\hbar}\, \left( -1+z \right) }}
\\ \noalign{\medskip}{\frac { \left( {\hbar}-1 \right)  \left( za_{
{2}}{\hbar}-a_{{1}} \right) }{ \left( a_{{1}}-a_{{2}} \right) {\hbar}\, \left( -1+z \right) }}&{\frac { \left( {\hbar}\,a_{{1}}-a_{
{2}} \right)  \left( z{\hbar}-1 \right) }{ \left( a_{{1}}-a_{{2}}
 \right) {\hbar}\, \left( -1+z \right) }}\end {array} \right],
\]
and 
\[
\Mon_{0+\eps} =  \left[ \begin {array}{cc} {\frac {-{\hbar}\,a_{{2}}+a_{{1}}}{{\hbar}\, \left( a_{{1}}-a_{{2}} \right) }}&-{\frac {{a_{{2}}}^{2}
 \left( {\hbar}-1 \right) }{ \left( a_{{1}}-a_{{2}} \right) {\hbar}\,a_{{1}}}}\\ \noalign{\medskip}{\frac {{a_{{1}}}^{2} \left( {
\hbar}-1 \right) }{{\hbar}\, \left( a_{{1}}-a_{{2}} \right) a_{{
2}}}}&{\frac {{\hbar}\,a_{{1}}-a_{{2}}}{{\hbar}\, \left( a_{{1}}
-a_{{2}} \right) }}\end {array} \right], 
\]
and we have
\[
\Mon_{0+\eps} = \lim_{z\to 0} \Mon_0, \ \ \Mon_{0-\eps} = \lim_{z\to \infty} \Mon_0
\]
The operators $\Mon_{\pm \eps}$ on $K_\sT(\sX)$ are the R-matrices for the infinitesimal slopes, and if we divide $\Mon_0$ by either of them, we obtain the wall operator in the basis of fixed points:
\[
\Glue =  \left[ \begin {array}{cc} {\frac { \left( -1+ \left( {{\hbar}}^{2}-{
\hbar}+1 \right) z \right) a_{{1}}-za_{{2}}{\hbar}+a_{{2}}}{ \left( 
a_{{1}}-a_{{2}} \right)  \left( -1+z \right) }}&{\frac {z \left( {\hbar}\,a_{{2}}-a_{{1}} \right) a_{{2}} \left( -1+{\hbar} \right) }{a_{{
1}} \left( a_{{1}}-a_{{2}} \right)  \left( -1+z \right) }}
\\ \noalign{\medskip}-{\frac {a_{{1}}z \left( -1+{\hbar} \right) 
 \left( {\hbar}\,a_{{1}}-a_{{2}} \right) }{a_{{2}} \left( a_{{1}}-a_{
{2}} \right)  \left( -1+z \right) }}&{\frac { \left( 1+ \left( -{{\hbar}}^{2}+{\hbar}-1 \right) z \right) a_{{2}}+a_{{1}} \left( {\hbar}
\,z-1 \right) }{ \left( a_{{1}}-a_{{2}} \right)  \left( -1+z \right) }
}\end {array} \right] 
\]
In the stable basis we get
\[
G_{stable} =  \left[ \begin {array}{cc} {\frac {-1+ \left( \hbar^{2}-h+1 \right) z}{-
1+z}}&-{\frac { \left( -1+h \right) \sqrt{\hbar}z}{-1+z}}
\\ \noalign{\medskip}-{\frac { \left( -1+h \right) \sqrt{\hbar}z}{-1+z}}&
{\frac {z\hbar-1}{-1+z}}\end {array} \right],
\]
or, if we express it as an operator from the stable envelopes with slope $-\eps$ to the slope $\eps$, we get the R-matrix for $\sX^! \cong T^*\P^1$.
\[
\Glue_{Stab^{[-\eps]} \to Stab^{[+\eps]}} =  \left[ \begin {array}{cc} {\frac {z{\hbar}-1}{-1+z}}&-{\frac {
\sqrt {{\hbar}} \left( {\hbar}-1 \right) z}{-1+z}}
\\ \noalign{\medskip}-{\frac {{\hbar}-1}{ \left( -1+z \right) 
\sqrt {{\hbar}}}}&{\frac {z{\hbar}-1}{-1+z}}\end {array}
 \right] 
\]

For the computation of the shift operator, we also need the stable envelopes for the opposite chamber:
\[
\Stab_{-\fC} =  \left[ \begin {array}{cc} \vartheta \left( {\frac {a_{{1}}}{a_{{2}}}}
 \right) &0\\ \noalign{\medskip}{\frac {\vartheta \left( {{\hbar}}^{-1
} \right) }{\vartheta \left( z{\hbar} \right) }\vartheta \left( {\frac {z
a_{{2}}{\hbar}}{a_{{1}}}} \right) }&\vartheta \left( {\frac {a_{{1}}}{
{\hbar}\,a_{{2}}}} \right) \end {array} \right] 
\]
The monodromy
\[
\Mon = \Stab_{\fC} \cdot \Stab_{\fC}^{-1}
\]
is now acting on $K_\sT(\sX)$.
The limit is
\[
\lim_{q \to 0} \Mon =  \left[ \begin {array}{cc} -{\frac { \left( a_{{1}}-a_{{2}} \right) 
\sqrt {{\hbar}}}{{\hbar}\,a_{{1}}-a_{{2}}}}&{\frac { \left( {
\hbar}\,za_{{1}}-a_{{2}} \right)  \left( {\hbar}-1 \right) }{
 \left( {\hbar}\,a_{{1}}-a_{{2}} \right)  \left( z{\hbar}-1
 \right) }}\\ \noalign{\medskip}{\frac { \left( {\hbar}-1 \right) 
 \left( za_{{2}}{\hbar}-a_{{1}} \right) }{ \left( {\hbar}\,a_{{1
}}-a_{{2}} \right)  \left( z{\hbar}-1 \right) }}&-{\frac { \left( -
1+z \right)  \left( z{{\hbar}}^{2}-1 \right)  \left( a_{{1}}-a_{{2}
} \right) \sqrt {{\hbar}}}{ \left( {\hbar}\,a_{{1}}-a_{{2}}
 \right)  \left( z{\hbar}-1 \right) ^{2}}}\end {array} \right],
\]
and this is the shift operator written in the stable basis for the dual variety.
Untwisting bu the stable envelopes of $\sX^!$ for different slopes on both sides, we obtain the shift operator, which is equal to the R-matrix:
\[
S_{Stab^{[-\eps], \sX^!} \to Stab^{[+\eps], \sX^!}} = 
 \left[ \begin {array}{cc} {\frac {a_{{1}}-a_{{2}}}{{\hbar}\,a_{{1}
}-a_{{2}}}}&-{\frac { \left( -1+{\hbar} \right) a_{{1}}}{ \left( {
\hbar}\,a_{{1}}-a_{{2}} \right) \sqrt {{\hbar}}}}
\\ \noalign{\medskip}-{\frac { \left( -1+{\hbar} \right) a_{{2}}}{
 \left( {\hbar}\,a_{{1}}-a_{{2}} \right) \sqrt {{\hbar}}}}&{
\frac {a_{{1}}-a_{{2}}}{{\hbar}\,a_{{1}}-a_{{2}}}}\end {array}
 \right] 
\]

\section{Hilbert scheme}
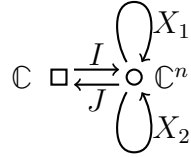
\begin{figure}[h]
    \centering
    \begin{tikzpicture}
    \draw [thick] (-0.1,-0.1)--(-0.1,0.1)--(0.1,0.1)--(0.1,-0.1)--cycle;
    \draw [thick] (1,0) circle (0.1);
    \draw [->, thick, rotate around = {90: (1,0)}] (1.2,0.1) to [out=30, in=90] (2,0) to [out=270, in=330] (1.2,-0.1);
    \draw [<-, thick, rotate around = {-90: (1,0)}] (1.2,0.1) to [out=30, in=90] (2,0) to [out=270, in=330] (1.2,-0.1);
    \draw [->, thick] (0.2,0.1) -- (0.8,0.1);
    \draw [<-, thick] (0.2,-0.1) -- (0.8,-0.1);
    \node at (-0.5,0) {$\C$};
    \node at (1.5,0) {$\C^n$};
    \node at (0.5,0.3) {$I$};
    \node at (0.5,-0.3) {$J$};
    \node at (1.5,0.7) {$X_1$};
    \node at (1.5,-0.7) {$X_2$};
    \end{tikzpicture}
    \caption{Hilbert scheme: $\{X_1, X_2, I, J: [X_1, X_2]+IJ=0\}//GL(n)$.}
    \label{fig:my_label}
\end{figure}

Let the character of the tautological bundle be
\[
V = \sum_{i=1}^n x_i.
\]
The tangent bundle as a function of Chern roots of the tautological bundle can be written as
\begin{multline}
T(x_1, ..., x_n) = V + t_1 t_2  \overline V -(1-t_1)(1-t_2)V \overline V = \\
= \sum_i (x_i+\frac{t_1 t_2}{x_i}) + \sum_{i,j}\left(
t_1 \frac{x_i}{x_j} + t_2 \frac{x_i}{x_j} - \frac{x_i}{x_j} - t_1 t_2 \frac{x_i}{x_j}
\right).
\end{multline}
To get fibers at a fixed point $\lambda$ we should substitute contents of boxes instead of variables $x_i$. Let the content of $i$-th box in $\lambda$ be $c_i \in \Z$.

\subsection{K-theory}
The abelianization technique developed by \cite{Shenfeld, OkBethe} gives the following procedure for construction of the off-shell K-theoretic stable envelopes $\Stab(\lambda)$: take the part of $T(x_1, ..., x_n)$ which is reprelling after specialization to the diagram $\lambda$:
\[
T(x_1, ..., x_n)_{<0} = \sum_{c_i<0} x_i + \sum_{c_i>0} \frac{t_1 t_2}{x_i} + \sum_{1+c_i-c_j<0} t_1 \frac{x_i}{x_j} + \sum_{-1+c_i-c_j<0} t_2 \frac{x_i}{x_j} - \sum_{c_i-c_j<0}  \frac{x_i}{x_j} - \sum_{c_i-c_j<0} t_1 t_2 \frac{x_i}{x_j},
\]
and 
\[
S_\lambda = \bigwedge\nolimits^\bullet T(x_1, ... , x_n).
\]
Explicitly, 
\[
S_\lambda(x_1, ... ,x_n) = \underset{x_1, ..., x_n}{\Sym}
\frac{\prod\limits_{c_i<0} (1-x_i)  \prod\limits_{c_i>0} (1-\frac{t_1 t_2}{x_i})  \prod\limits_{1+c_i-c_j<0} (1-t_1 \frac{x_i}{x_j})  \prod\limits_{-1+c_i-c_j<0} (1-t_2 \frac{x_i}{x_j})}
{\prod\limits_{c_i-c_j<0}  (1-\frac{x_i}{x_j}) \prod\limits_{c_i-c_j<0} (1-t_1 t_2 \frac{x_i}{x_j})}
\]

\subsection{Elliptic cohomology}

The polarization is given by
\[
T^{1/2} = V - (1-t_1) V\overline V = \sum_i x_i + \sum_{i,j} \left(t_1 \frac{x_i}{x_j} - \frac{x_i}{x_j}\right)
\]

Let us start with $n=2$ and the diagram $\lambda = [2]$.
At this point, we have 
\[
x_1 = 1, x_2 = t_1^{-1}.
\]
Since we are interested in off-shell stable envelopes as functions of $x_i$ up to overall normalization, we will keep only terms which depend nontrivially on them.
Then we compute
\[
T^{1/2}_{>0} = 2t_1 + t_1 \frac{x_1}{x_2}-\frac{x_1}{x_2}
\]
\[
T^{1/2}_{<0} = x_2 - \frac{x_2}{x_1}
\]
\[
T^{1/2}_{=0} = x_1 - 2 + t_1 \frac{x_2}{x_1}
\]
The Shenfeld factor is just
\[
\Sh = \vartheta(T^{1/2}_{<0}) \cdot \vartheta(\hbar^{-1} T^{1/2, \vee}_{>0}) = \frac{\vartheta(t_2)^2 \vartheta(x_2) \vartheta(t_2 x_2/x_1)}{\vartheta(x_2/x_1) \vartheta(t_1 t_2 x_2/x_1)}
\]
However, the degree of this section is not invariant under permutations. The $T_{=0}^{1/2}$ part should be refined with parameters $z$ and $\hbar$, and this refinement is uniquely fixed by the quasiperiods, and we obtain the "root" part:
\[
\Rt = \frac{\vartheta(x_1 z^2)}{\vartheta(z^2) } \frac{\vartheta(t_1 \frac{x_2}{x_1} z)}{ \vartheta(z)}
\]

Consider now the diagram $\lambda = [1,1]$. Analogously, we have
\[
x_1 = 1, x_2 = t_2^{-1}
\]
\[
T_{<0} = -\frac{x_1}{x_2}
\]
\[
T_{>0} = 2t_1 + x_2 + t_1 \frac{x_2}{x_1} - \frac{x_2}{x_1}.
\]
\[
T_{=0} = x_1 - 2 + t_1 \frac{x_1}{x_2}.
\]
\[
\Sh = \vartheta(T^{1/2}_{<0}) \cdot \vartheta(\hbar^{-1} T^{1/2, \vee}_{>0}) = \frac{\vartheta(t_2)^2 \vartheta(t_1 t_2/x_2) \vartheta(t_2 x_1/x_2)}{\vartheta(x_1/x_2) \vartheta(t_1 t_2 x_1/x_2)}
\]
The "root" part in this case is
\[
\Rt = \frac{\vartheta(x_1 z^2 t_1 t_2)}{\vartheta(z^2 t_1 t_2)} \frac{\vartheta(t_2 \frac{x_2}{x_1} z)}{\vartheta(z)}.
\]
Note that the $x_1$-term was deformed by $\hbar$ since otherwise after shifting $x_i$ by $q$ we would obtain $\hbar$ as a quasiperiod.

For $n\geq 4$ there are non single-hook Young diagrams, and the corresponding fixed points in the abelianization of the Hilbert scheme are not isolated, and we need to use special tricks by introducing larger equivariant torus and deforming the cycles to make them equivariant with respect to the larger torus, see \cite{SmirnovElliptic}.

\subsection{Factorization}

For $n=2$ we get the following matrix of elliptic stable envelopes:
\[
\Stab =  \left[ \begin {array}{cc} \vartheta \left( t_{{2}} \right) \vartheta
 \left( t_{{2}}^{2} \right) &{\frac { \vartheta \left( t_{{2}}
 \right)^{2}\vartheta \left( t_{{1}}t_{{2}} \right) \vartheta ( { 
{t_{{2}}z}/{t_{{1}}}} )  }{\vartheta
 \left( t_{{1}} \right) \vartheta \left( z \right) }}+{\frac {\vartheta \left( t_{{2}} \right) 
\vartheta \left( t_{{1}}t_{{2}} \right) \vartheta ( t_{{2}}{z}^{2}) \vartheta \left( t_{{1}}t_{{2}}z \right) \vartheta \left( {
 {t_{{2}}}/{t_{{1}}}} \right) }{\vartheta \left( t_{{1}
}^{-1} \right) \vartheta \left( {z}^{2}t_{{1}}t_{{2}} \right) \vartheta
 \left( z \right) }  }
\\ \noalign{\medskip}0&\vartheta \left( t_{{2}} \right) \vartheta \left( {
 {t_{{2}}}/{t_{{1}}}} \right)  \end {array} \right],
\]
In the limit to the wall $\sw = 1/2$ we get
\begin{multline}
\lim_{q\to 0} \Stab(a, zq^{1/2}) =   \left[ \begin {array}{cc} {\frac { \left( -1+t_{{2}} \right) ^{2}
 \left( t_{{2}}+1 \right) }{{t_{{2}}}^{3/2}}}&{\frac { \left( {z}^{2}{
t_{{2}}}^{2}-1 \right)  \left( -1+t_{{2}} \right)  \left( t_{{1}}t_{{2
}}-1 \right) \sqrt {t_{{1}}}}{ \left( {z}^{2}t_{{1}}t_{{2}}-1 \right) 
{t_{{2}}}^{2}}}\\ \noalign{\medskip}0&-{\frac { \left( -1+t_{{2}}
 \right)  \left( -t_{{2}}+t_{{1}} \right) }{t_{{2}}\sqrt {t_{{1}}}}}
\end {array} \right]  = \\
=
\left[ \begin {array}{cc} 1&{\frac {t_{{1}}{z}^{2} \left( t_{{1}}t_{{
2}}-1 \right) }{{z}^{2}t_{{1}}t_{{2}}-1}}\\ \noalign{\medskip}0&1
\end {array} \right] \cdot  \left[ \begin {array}{cc} {\frac { \left( -1+t_{{2}} \right) ^{2}
 \left( t_{{2}}+1 \right) }{{t_{{2}}}^{3/2}}}&{\frac { \left( -1+t_{{2
}} \right)  \left( t_{{1}}t_{{2}}-1 \right) \sqrt {t_{{1}}}}{{t_{{2}}}
^{2}}}\\ \noalign{\medskip}0&-{\frac { \left( -1+t_{{2}} \right) 
 \left( -t_{{2}}+t_{{1}} \right) }{t_{{2}}\sqrt {t_{{1}}}}}
\end {array} \right] = \\
=
 \left[ \begin {array}{cc} 1&{\frac {t_{{1}}t_{{2}}-1}{t_{{2}} \left( 
{z}^{2}t_{{1}}t_{{2}}-1 \right) }}\\ \noalign{\medskip}0&1\end {array}
 \right]  \cdot  \left[ \begin {array}{cc} {\frac { \left( -1+t_{{2}} \right) ^{2}
 \left( t_{{2}}+1 \right) }{{t_{{2}}}^{3/2}}}&{\frac { \left( t_{{1}}t
_{{2}}-1 \right)  \left( -1+t_{{2}} \right) }{t_{{2}}\sqrt {t_{{1}}}}}
\\ \noalign{\medskip}0&-{\frac { \left( -1+t_{{2}} \right)  \left( -t_
{{2}}+t_{{1}} \right) }{t_{{2}}\sqrt {t_{{1}}}}}\end {array} \right] 
\end{multline}

The first factorization corresponds to the slope $\sw = 1/2 + \eps$, the second to $\sw = 1/2 - \eps$.

\subsection{Difference equations}

Let us compute for $n=2$.
The operator of classical multiplication
\[
\cO(1) =  \left[ \begin {array}{cc} {t_{{2}}}^{-1}&0\\ \noalign{\medskip}0&{t_{
{1}}}^{-1}\end {array} \right]
\]
The K-theoretic off-shell stable envelopes are (so-called  {\it off-shell Bethe vector}) given by
\[
\Stab^{K}_+([2]) = -{\frac { \left( x_{{2}}-1 \right)  \left( t_{{2}}-1 \right) ^{2}
 \left( t_{{2}}x_{{2}}-x_{{1}} \right) }{ \left( -x_{{2}}+x_{{1}}
 \right)  \left( t_{{1}}t_{{2}}x_{{2}}-x_{{1}} \right) }}-{\frac {
 \left( x_{{1}}-1 \right)  \left( t_{{2}}-1 \right) ^{2} \left( t_{{2}
}x_{{1}}-x_{{2}} \right) }{ \left( x_{{2}}-x_{{1}} \right)  \left( t_{
{1}}t_{{2}}x_{{1}}-x_{{2}} \right) }}
\]
\[
\Stab^{K}_+([1,1]) = {\frac { \left( t_{{1}}t_{{2}}-x_{{2}} \right)  \left( t_{{2}}-1
 \right) ^{2} \left( t_{{2}}x_{{1}}-x_{{2}} \right) }{ \left( -x_{{2}}
+x_{{1}} \right)  \left( t_{{1}}t_{{2}}x_{{1}}-x_{{2}} \right) }}+{
\frac { \left( t_{{1}}t_{{2}}-x_{{1}} \right)  \left( t_{{2}}-1
 \right) ^{2} \left( t_{{2}}x_{{2}}-x_{{1}} \right) }{ \left( x_{{2}}-
x_{{1}} \right)  \left( t_{{1}}t_{{2}}x_{{2}}-x_{{1}} \right) }}
\]
\[
\Stab^{K}_-([2]) = {\frac { \left( -1+t_{{1}} \right) ^{2} \left( t_{{1}}x_{{1}}-x_{{2}}
 \right)  \left( t_{{1}}t_{{2}}-x_{{2}} \right) }{ \left( -x_{{2}}+x_{
{1}} \right)  \left( t_{{1}}t_{{2}}x_{{1}}-x_{{2}} \right) }}+{\frac {
 \left( -1+t_{{1}} \right) ^{2} \left( t_{{1}}x_{{2}}-x_{{1}} \right) 
 \left( t_{{1}}t_{{2}}-x_{{1}} \right) }{ \left( x_{{2}}-x_{{1}}
 \right)  \left( t_{{1}}t_{{2}}x_{{2}}-x_{{1}} \right) }}
\]
\[
\Stab^{K}_-([1,1]) = -{\frac { \left( x_{{2}}-1 \right)  \left( -1+t_{{1}} \right) ^{2}
 \left( t_{{1}}x_{{2}}-x_{{1}} \right) }{ \left( -x_{{2}}+x_{{1}}
 \right)  \left( t_{{1}}t_{{2}}x_{{2}}-x_{{1}} \right) }}-{\frac {
 \left( x_{{1}}-1 \right)  \left( -1+t_{{1}} \right) ^{2} \left( t_{{1
}}x_{{1}}-x_{{2}} \right) }{ \left( x_{{2}}-x_{{1}} \right)  \left( t_
{{1}}t_{{2}}x_{{1}}-x_{{2}} \right) }}.
\]
Here we did not care about normalization and shift by integer slope.

Stable envelopes for the slope $0<\sw<1/2$ can be obtained by restrictions of the off-shell stable envelopes and the Gram-Schmidt procedure:
\[
\Stab^{K}_+ =  \left[ \begin {array}{cc} {\frac { \left( t_{{2}}-1 \right) ^{2}
 \left( t_{{2}}+1 \right) }{t_{{2}}{t_{{1}}}^{2}}}&{\frac { \left( t_{
{2}}-1 \right)  \left( t_{{1}}t_{{2}}-1 \right) }{t_{{2}}{t_{{1}}}^{2}
}}\\ \noalign{\medskip}0&-{\frac { \left( -t_{{2}}+t_{{1}} \right) 
 \left( t_{{2}}-1 \right) }{t_{{2}}{t_{{1}}}^{2}}}\end {array}
 \right] 
\]
\[
\Stab^K_- =  \left[ \begin {array}{cc} {\frac { \left( -t_{{2}}+t_{{1}} \right) 
 \left( -1+t_{{1}} \right) }{{t_{{1}}}^{3}}}&0\\ \noalign{\medskip}{
\frac { \left( t_{{1}}t_{{2}}-1 \right)  \left( -1+t_{{1}} \right) }{{
t_{{1}}}^{3}}}&{\frac { \left( -1+t_{{1}} \right) ^{2} \left( 1+t_{{1}
} \right) }{{t_{{1}}}^{3}}}\end {array} \right] 
\]

As in the case of the projective space,
the gluing matrix in degree 1 can be computed directly by localization
\[
 \Glue_{degree 1} = \left[ \begin {array}{cc} {\frac { \left( t_{{1}}t_{{2}}-1 \right) 
 \left( -1+t_{{1}} \right)  \left( t_{{2}}+1 \right) }{\sqrt {t_{{1}}}
\sqrt {t_{{2}}} \left( -t_{{2}}+t_{{1}} \right) }}&-{\frac { \left( t_
{{1}}t_{{2}}-1 \right)  \left( {t_{{2}}}^{2}-1 \right) }{\sqrt {t_{{1}
}}\sqrt {t_{{2}}} \left( -t_{{2}}+t_{{1}} \right) }}
\\ \noalign{\medskip}{\frac { \left( t_{{1}}t_{{2}}-1 \right)  \left( 
{t_{{1}}}^{2}-1 \right) }{\sqrt {t_{{1}}}\sqrt {t_{{2}}} \left( -t_{{2
}}+t_{{1}} \right) }}&-{\frac { \left( t_{{1}}t_{{2}}-1 \right) 
 \left( 1+t_{{1}} \right)  \left( t_{{2}}-1 \right) }{\sqrt {t_{{1}}}
\sqrt {t_{{2}}} \left( -t_{{2}}+t_{{1}} \right) }}\end {array}
 \right],
\]
however, it does not suffice for the full gluing operator.

We can compute the solution to the quantum difference equation (called {\it capping}) before the diffence  equation itself.
Capping can be computed directly as a series in $z$ by using the approach in \cite{OkBethe}, which relates the relative boundary conditions with insertions of the tautological classes of K-theoretic off-shell stable envelopes:
\[
\Psi(z) =  \left[ \begin {array}{cc} 1-{\frac { \left( t_{{2}}+1 \right) 
 \left( -1+t_{{1}} \right) q \left( t_{{1}}t_{{2}}-1 \right) }{
 \left( -t_{{2}}+t_{{1}} \right)  \left( q-1 \right) }{\frac {1}{
\sqrt {t_{{1}}}}}{\frac {1}{\sqrt {t_{{2}}}}}}z + \ldots &-{\frac {q \left( t_
{{1}}t_{{2}}-1 \right)  \left( {t_{{1}}}^{2}-1 \right) }{ \left( qt_{{
2}}-t_{{1}} \right)  \left( -t_{{2}}+t_{{1}} \right) }\sqrt {t_{{2}}}{
\frac {1}{\sqrt {t_{{1}}}}}}z + \ldots\\ \noalign{\medskip}{\frac {q \left( t
_{{1}}t_{{2}}-1 \right)  \left( {t_{{2}}}^{2}-1 \right) }{ \left( -t_{
{2}}+t_{{1}} \right)  \left( t_{{1}}q-t_{{2}} \right) }\sqrt {t_{{1}}}
{\frac {1}{\sqrt {t_{{2}}}}}}z + \ldots&1+{\frac { \left( t_{{2}}-1 \right) q
 \left( t_{{1}}t_{{2}}-1 \right)  \left( 1+t_{{1}} \right) }{ \left( -
t_{{2}}+t_{{1}} \right)  \left( q-1 \right) }{\frac {1}{\sqrt {t_{{1}}
}}}{\frac {1}{\sqrt {t_{{2}}}}}}z + \ldots \end {array} \right]
\]
It satisfies the equation
\[
\Psi(qz) \cO(1) = M(z) \Psi(z),
\]
from which we can compute $M(z)$.  Unfortunately, the expression fot $M(z)$ is too long to be published here, but we will publish the wall operators in the stable basis.
The wall operators can be computed as
\[
\B_0 (zq^{-1}) = \lim_{q\to \infty} M(zq^{-1}) \cO(-1),
\]
\[
\B_{1/2} (zq^{-1/2}) = \lim_{q\to \infty} M(zq^{-1/2}) \cO(-1),
\]
and then
\[
M(z) = \B_0 \B_{1/2} \cO(1).
\]
In the stable basis for $\B_\sw$ from $\Stab^{[\sw-\eps]}$ to $\Stab^{[\sw + \eps]}$, they depend only on $z, q, \hbar$:
\[
\B_0 =  \left[ \begin {array}{cc} {\frac { \left( \sqrt {t_{{1}}}\sqrt {t_{{2
}}}qz-1 \right) {t_{{1}}}^{3/2}{t_{{2}}}^{3/2} \left( {q}^{2}{z}^{2}-1
 \right) }{ \left( qz+\sqrt {t_{{1}}}\sqrt {t_{{2}}} \right)  \left( 
\sqrt {t_{{1}}}\sqrt {t_{{2}}}-qz \right) ^{2}}}&{\frac {qzt_{{2}}
 \left( \sqrt {t_{{1}}}\sqrt {t_{{2}}}qz-1 \right) t_{{1}} \left( t_{{
1}}t_{{2}}-1 \right) }{ \left( qz+\sqrt {t_{{1}}}\sqrt {t_{{2}}}
 \right)  \left( \sqrt {t_{{1}}}\sqrt {t_{{2}}}-qz \right) ^{2}}}
\\ \noalign{\medskip}{\frac {qzt_{{2}} \left( \sqrt {t_{{1}}}\sqrt {t_
{{2}}}qz-1 \right) t_{{1}} \left( t_{{1}}t_{{2}}-1 \right) }{ \left( q
z+\sqrt {t_{{1}}}\sqrt {t_{{2}}} \right)  \left( \sqrt {t_{{1}}}\sqrt 
{t_{{2}}}-qz \right) ^{2}}}&{\frac { \left( \sqrt {t_{{1}}}\sqrt {t_{{
2}}}qz-1 \right) {t_{{1}}}^{3/2}{t_{{2}}}^{3/2} \left( {q}^{2}{z}^{2}-
1 \right) }{ \left( qz+\sqrt {t_{{1}}}\sqrt {t_{{2}}} \right)  \left( 
\sqrt {t_{{1}}}\sqrt {t_{{2}}}-qz \right) ^{2}}}\end {array} \right] 
\]
\[
\B_{1/2} =  \left[ \begin {array}{cc} {\frac {t_{{2}}t_{{1}} \left( {z}^{2}q-1
 \right) }{{z}^{2}q-t_{{1}}t_{{2}}}}&-{\frac {t_{{1}} \left( t_{{1}}t_
{{2}}-1 \right) }{{z}^{2}q-t_{{1}}t_{{2}}}}\\ \noalign{\medskip}-{
\frac {{z}^{2}qt_{{2}} \left( t_{{1}}t_{{2}}-1 \right) }{{z}^{2}q-t_{{
1}}t_{{2}}}}&{\frac {t_{{2}}t_{{1}} \left( {z}^{2}q-1 \right) }{{z}^{2
}q-t_{{1}}t_{{2}}}}\end {array} \right] 
\]
are the R-matrices for $\sX^! \cong \Hilb(\C^2, 2)$ and $(\sX^!)^{\Z/2\Z} \cong \Hilb(T^*\P^1, 1)$ respectively.

Now let us obtain the wall operators from the monodromy.
The matrices of elliptic stable envelopes for $\sX$ and $\sX_{\flop}$ are
\[
\Stab =  \left[ \begin {array}{cc} \vartheta \left( t_{{2}} \right) \vartheta
 \left( t_{{2}}^{2} \right) &{\frac { \vartheta \left( t_{{2}}
 \right)^{2}\vartheta \left( t_{{1}}t_{{2}} \right) \vartheta ( { 
{t_{{2}}z}/{t_{{1}}}} )  }{\vartheta
 \left( t_{{1}} \right) \vartheta \left( z \right) }}+{\frac {\vartheta \left( t_{{2}} \right) 
\vartheta \left( t_{{1}}t_{{2}} \right) \vartheta ( t_{{2}}{z}^{2}) \vartheta \left( t_{{1}}t_{{2}}z \right) \vartheta \left( {
 {t_{{2}}}/{t_{{1}}}} \right) }{\vartheta \left( t_{{1}
}^{-1} \right) \vartheta \left( {z}^{2}t_{{1}}t_{{2}} \right) \vartheta
 \left( z \right) }  }
\\ \noalign{\medskip}0&\vartheta \left( t_{{2}} \right) \vartheta \left( {
 {t_{{2}}}/{t_{{1}}}} \right)  \end {array} \right],
\]
\[
\Stab_{\flop} =  \left[ \begin {array}{cc} {\frac {\vartheta \left( t_{{2}} \right) 
\vartheta \left( {z}^{2}t_{{1}}^{2}t_{{2}}^{2} \right) \vartheta \left( 
t_{{1}}t_{{2}}z \right)  \vartheta \left( { {t_{{2}}}/{t_{{1}}}}
 \right) }{\vartheta \left( {z}^{2}t_{{1}}t_{{2}} \right) 
\vartheta \left( z \right) } }&0\\ \noalign{\medskip}{\frac {\vartheta \left( t_{{2}}
 \right) \vartheta \left( t_{{1}}t_{{2}} \right) \vartheta \left( t_{{1}}^
{2}t_{{2}}{z}^{2} \right) \vartheta \left( { {t_{{2}}}/{t_{{1}}}
} \right) }{\vartheta \left( t_{{1}}^{-1} \right) 
\vartheta \left( {z}^{2} \right) } }+{\frac {  \vartheta \left( t_{{2}} \right)^{2
}\vartheta \left( t_{{1}}t_{{2}} \right) \vartheta \left( {z}^{2}t_{{1}}t_{{
2}} \right) \vartheta \left( z t_{{1}}^{2} \right) }{\vartheta \left( t_{{1
}} \right) \vartheta \left( {z}^{2} \right) \vartheta \left( z \right) }}&{
\frac {\vartheta \left( t_{{2}} \right) \vartheta \left( t_{{2}}^{2}
 \right) \vartheta \left( {z}^{2}t_{{1}}t_{{2}} \right) \vartheta \left( t_{
{1}}t_{{2}}z \right) }{\vartheta \left( {z}^{2} \right) \vartheta \left( z
 \right) }}\end {array} \right] 
\]

The walls are located at 
\[
\Z \bigcup \left(\Z+\frac{1}{2}\right) \subset \R \cong \Pic(\Hilb(\C^2,2)) \otimes \R.
\]
It is enough to compute the wall operators $\B_0, \B_{1/2}$, since
\[
\B_{\sw + n} = \cO(n) \B_{\sw} \cO(-n).
\]

The limits of the monodromy to the walls are
\[
\Mon_0 =  \left[ \begin {array}{cc} {\frac { \left( t_{{1}}t_{{2}}z-1 \right) 
 \left( {z}^{2}t_{{1}}t_{{2}}-1 \right)  \left( {t_{{1}}}^{2}-1
 \right) }{\sqrt {t_{{2}}}{t_{{1}}}^{3/2} \left( t_{{1}}-t_{{2}}
 \right)  \left( -1+z \right) ^{2} \left( z+1 \right) }}&{\frac {
 \left( t_{{1}}t_{{2}}-1 \right)  \left( t_{{2}}z+1 \right)  \left( zt
_{{1}}-1 \right)  \left( t_{{1}}t_{{2}}z-1 \right) }{\sqrt {t_{{1}}}{t
_{{2}}}^{3/2} \left( t_{{1}}-t_{{2}} \right)  \left( -1+z \right) ^{2}
 \left( z+1 \right) }}\\ \noalign{\medskip}-{\frac { \left( t_{{1}}t_{
{2}}-1 \right)  \left( t_{{2}}z-1 \right)  \left( zt_{{1}}+1 \right) 
 \left( t_{{1}}t_{{2}}z-1 \right) }{\sqrt {t_{{2}}}{t_{{1}}}^{3/2}
 \left( t_{{1}}-t_{{2}} \right)  \left( -1+z \right) ^{2} \left( z+1
 \right) }}&-{\frac { \left( {t_{{2}}}^{2}-1 \right)  \left( {z}^{2}t_
{{1}}t_{{2}}-1 \right)  \left( t_{{1}}t_{{2}}z-1 \right) }{\sqrt {t_{{
1}}}{t_{{2}}}^{3/2} \left( t_{{1}}-t_{{2}} \right)  \left( -1+z
 \right) ^{2} \left( z+1 \right) }}\end {array} \right],
\]
\[
\Mon_{0 + \eps} =  \left[ \begin {array}{cc} {\frac {{t_{{1}}}^{2}-1}{\sqrt {t_{{2}}}{t_
{{1}}}^{3/2} \left( t_{{1}}-t_{{2}} \right) }}&{\frac {t_{{1}}t_{{2}}-
1}{\sqrt {t_{{1}}}{t_{{2}}}^{3/2} \left( t_{{1}}-t_{{2}} \right) }}
\\ \noalign{\medskip}-{\frac {t_{{1}}t_{{2}}-1}{\sqrt {t_{{2}}}{t_{{1}
}}^{3/2} \left( t_{{1}}-t_{{2}} \right) }}&-{\frac {{t_{{2}}}^{2}-1}{
\sqrt {t_{{1}}}{t_{{2}}}^{3/2} \left( t_{{1}}-t_{{2}} \right) }}
\end {array} \right],
\]
\[
\Mon_{1/2} =  \left[ \begin {array}{cc} {\frac { \left( {z}^{2}t_{{1}}t_{{2}}-1
 \right)  \left( {t_{{1}}}^{2}-1 \right) }{{t_{{1}}}^{5/2}{t_{{2}}}^{3
/2} \left( t_{{1}}-t_{{2}} \right)  \left( {z}^{2}-1 \right) }}&{
\frac { \left( {z}^{2}{t_{{2}}}^{2}-1 \right)  \left( t_{{1}}t_{{2}}-1
 \right) }{\sqrt {t_{{1}}} \left( {z}^{2}-1 \right)  \left( t_{{1}}-t_
{{2}} \right) {t_{{2}}}^{7/2}}}\\ \noalign{\medskip}-{\frac { \left( t
_{{1}}t_{{2}}-1 \right)  \left( {z}^{2}{t_{{1}}}^{2}-1 \right) }{
\sqrt {t_{{2}}} \left( t_{{1}}-t_{{2}} \right) {t_{{1}}}^{7/2} \left( 
{z}^{2}-1 \right) }}&-{\frac { \left( {z}^{2}t_{{1}}t_{{2}}-1 \right) 
 \left( {t_{{2}}}^{2}-1 \right) }{{t_{{1}}}^{3/2}{t_{{2}}}^{5/2}
 \left( t_{{1}}-t_{{2}} \right)  \left( {z}^{2}-1 \right) }}
\end {array} \right],
\]
\[
\Mon_{1/2+\eps} =  \left[ \begin {array}{cc} {\frac {{t_{{1}}}^{2}-1}{{t_{{1}}}^{5/2}{t_
{{2}}}^{3/2} \left( t_{{1}}-t_{{2}} \right) }}&{\frac {t_{{1}}t_{{2}}-
1}{\sqrt {t_{{1}}} \left( t_{{1}}-t_{{2}} \right) {t_{{2}}}^{7/2}}}
\\ \noalign{\medskip}-{\frac {t_{{1}}t_{{2}}-1}{\sqrt {t_{{2}}}
 \left( t_{{1}}-t_{{2}} \right) {t_{{1}}}^{7/2}}}&-{\frac {{t_{{2}}}^{
2}-1}{{t_{{1}}}^{3/2}{t_{{2}}}^{5/2} \left( t_{{1}}-t_{{2}} \right) }}
\end {array} \right].
\]

We then have
\[
\B_\sw(z q^{-\sw} \sqrt{t_1 t_2}) = \Mon_\sw \cdot \Mon_{\sw + \eps}^{-1}.
\]

The operator of the quantum difference equation can now be computed as
\[
M(z) = \B_0 \B_{1/2} \cO(1).
\]



\titleformat{\chapter}[display]
{\normalfont\bfseries\filcenter}{}{0pt}{\large\bfseries\filcenter{#1}}  
\titlespacing*{\chapter}
  {0pt}{0pt}{30pt}

\begin{tiny}
	\setlength\bibitemsep{\baselineskip}  
	\begingroup
	\setstretch{1.0}
	\printbibliography[title={References}]

@misc{KononovSmirnov2,
      title={Pursuing quantum difference equations II: 3D-mirror symmetry}, 
      author={Yakov Kononov and Andrey Smirnov},
      year={2020},
      eprint={2008.06309},
      archivePrefix={arXiv},
      primaryClass={math.AG}
}

@misc{OkInductive1,
      title={Inductive construction of stable envelopes and applications, I. Actions of tori. Elliptic cohomology and K-theory}, 
      author={Andrei Okounkov},
      year={2020},
      eprint={2007.09094},
      archivePrefix={arXiv},
      primaryClass={math.AG}
}

@misc{OkInductive2,
      title={Inductive construction of stable envelopes and applications, II. Nonabelian actions. Integral solutions and monodromy of quantum difference equations}, 
      author={Andrei Okounkov},
      year={2020},
      eprint={2010.13217},
      archivePrefix={arXiv},
      primaryClass={math.AG}
}

@article
{MO,
author = {Maulik, Davesh and Okounkov, Andrei},
year = {2012},
month = {11},
pages = {},
title = {{Quantum Groups and Quantum Cohomology}},
volume = {408},
journal = {Astérisque},
doi = {10.24033/ast.1074}
}

@article {InstR,
    AUTHOR = {Smirnov, Andrey},
     title = {{On the instanton $R$-matrix}},
   JOURNAL = {Comm. Math. Phys.},
  FJOURNAL = {Communications in Mathematical Physics},
    VOLUME = {345},
      YEAR = {2016},
    NUMBER = {3},
     PAGES = {703--740},
      ISSN = {0010-3616},
   MRCLASS = {57R56 (14C05 17B65 33C67)},
  MRNUMBER = {3519581},
       DOI = {10.1007/s00220-016-2686-8},
       URL = {http://dx.doi.org/10.1007/s00220-016-2686-8},
}

@ARTICLE{KOO,
       author = {{Kononov}, Ya. and {Okounkov}, A. and {Osinenko}, A.},
        title = {{The 2-leg vertex in K-theoretic DT theory}},
      journal = {arXiv e-prints},
         year = 2019,
        month = may,
          eid = {arXiv:1905.01523},
        pages = {arXiv:1905.01523},
archivePrefix = {arXiv},
       eprint = {1905.01523},
 primaryClass = {math-ph},
       adsurl = {https://ui.adsabs.harvard.edu/abs/2019arXiv190501523K}
}

@inbook{pcmilect,
author = {Okounkov, Andrei},
year = {2017},
month = {12},
pages = {251-380},
title = {{Lectures on K-theoretic computations in enumerative geometry}},
isbn = {9781470435745},
doi = {10.1090/pcms/024/05}
}

@article{FR,
      author         = "Frenkel, I. B. and Reshetikhin, N. {\relax Yu}.",
      title          = {Quantum affine algebras and holonomic difference
                        equations},
      journal        = "Commun. Math. Phys.",
      volume         = "146",
      year           = "1992",
      pages          = "1-60",
      doi            = "10.1007/BF02099206",
      SLACcitation   = "%%CITATION = CMPHA,146,1;%%"
}

@article {EV,
    AUTHOR = {Etingof, P. and Varchenko, A.},
     title = {{Dynamical Weyl groups and applications}},
   JOURNAL = {Adv. Math.},
  FJOURNAL = {Advances in Mathematics},
    VOLUME = {167},
      YEAR = {2002},
    NUMBER = {1},
     PAGES = {74--127},
      ISSN = {0001-8708},
     CODEN = {ADMTA4},
   MRCLASS = {17B10 (17B20 17B37 39A12 81R50)},
  MRNUMBER = {1901247 (2003d:17004)},
MRREVIEWER = {Anjan Kundu},
       DOI = {10.1006/aima.2001.2034}
}

@preamble{
   "\def\cprime{$'$} "
}

@article {FRT,
    AUTHOR = {Reshetikhin, N. Yu. and Takhtadzhyan, L. A. and Faddeev, L.
              D.},
     title = {{Quantization of Lie groups and Lie algebras}},
   JOURNAL = {Algebra i Analiz},
  FJOURNAL = {Algebra i Analiz},
    VOLUME = {1},
      YEAR = {1989},
    NUMBER = {1},
     PAGES = {178--206},
      ISSN = {0234-0852},
   MRCLASS = {17B65 (17B35 22E46 58F07 81D07 82A69)},
  MRNUMBER = {1015339 (90j:17039)},
MRREVIEWER = {Ya. S. So{\u\i}bel{\cprime}man},
}

@article {Nak1,
    AUTHOR = {Nakajima, Hiraku},
     title = {{Quiver varieties and Kac-Moody algebras}},
   JOURNAL = {Duke Math. J.},
  FJOURNAL = {Duke Mathematical Journal},
    VOLUME = {91},
      YEAR = {1998},
    NUMBER = {3},
     PAGES = {515--560},
      ISSN = {0012-7094},
     CODEN = {DUMJAO},
   MRCLASS = {17B67 (14D25 16G20 17B35 53C25 58F05)},
  MRNUMBER = {1604167 (99b:17033)},
MRREVIEWER = {Michael M. Kapranov},
       DOI = {10.1215/S0012-7094-98-09120-7},
       URL = {http://dx.doi.org/10.1215/S0012-7094-98-09120-7},
}

@article {SchifVas,
    AUTHOR = {Schiffmann, Olivier and Vasserot, Eric},
     TITLE = {{The elliptic Hall algebra and the $K$-theory of the
              Hilbert scheme of $\Bbb A^2$}},
   JOURNAL = {Duke Math. J.},
  FJOURNAL = {Duke Mathematical Journal},
    VOLUME = {162},
      YEAR = {2013},
    NUMBER = {2},
     PAGES = {279--366},
      ISSN = {0012-7094},
     CODEN = {DUMJAO},
   MRCLASS = {19E08 (14C05 14C35 14F05 20C08)},
  MRNUMBER = {3018956},
MRREVIEWER = {Jens Hornbostel},
       DOI = {10.1215/00127094-1961849},
       URL = {http://dx.doi.org/10.1215/00127094-1961849},
}

@article{Pandharipande,
   title={13/2 ways of counting curves},
   ISBN={9781107279544},
   url={http://dx.doi.org/10.1017/CBO9781107279544.007},
   DOI={10.1017/cbo9781107279544.007},
   journal={Moduli Spaces},
   publisher={Cambridge University Press},
   author={Pandharipande, R. and Thomas, R. P.},
   editor={Brambila-Paz, Leticia and Newstead, Peter and Thomas, Richard P. W. and Garcia-Prada, OscarEditors},
   pages={282–333}
}

@article{OS,
      author         = "Okounkov, Andrei and Smirnov, Andrey",
      title = {{Quantum difference equation for Nakajima varieties}},
      year           = "2016",
      eprint         = "1602.09007",
      archivePrefix  = "arXiv",
      primaryClass   = "math-ph",
      JOURNAL = {ArXiv: 1602.09007}
}

@article{AOF,
author = {Aganagic, Mina and Frenkel, Edward and Okounkov, Andrei},
year = {2017},
month = {01},
pages = {},
title = {{Quantum q-Langlands Correspondence}},
volume = {79},
journal = {Transactions of the Moscow Mathematical Society},
doi = {10.1090/mosc/278}
}

@article {GL,
    AUTHOR = {Givental, Alexander and Lee, Yuan-Pin},
     TITLE = {{Quantum $K$-theory on flag manifolds, finite-difference
              Toda lattices and quantum groups}},
   JOURNAL = {Invent. Math.},
  FJOURNAL = {Inventiones Mathematicae},
    VOLUME = {151},
      YEAR = {2003},
    NUMBER = {1},
     PAGES = {193--219},
      ISSN = {0020-9910},
     CODEN = {INVMBH},
   MRCLASS = {14N35 (14C35 14M15 17B37 37K20 39A70 53D45)},
  MRNUMBER = {1943747 (2004g:14063)},
MRREVIEWER = {Domenico Fiorenza},
       DOI = {10.1007/s00222-002-0250-y},
       URL = {http://dx.doi.org/10.1007/s00222-002-0250-y},
}

@book {Shenfeld,
	AUTHOR = {Shenfeld, Daniel},
	title = {{Abelianization of stable envelopes in symplectic resolutions}},
	NOTE = {Thesis (Ph.D.)--Princeton University},
	PUBLISHER = {ProQuest LLC, Ann Arbor, MI},
	YEAR = {2013},
	PAGES = {75},
	ISBN = {978-1303-45685-5},
	MRCLASS = {Thesis},
	MRNUMBER = {3192995},
	URL =
	{http://gateway.proquest.com/openurl?url_ver=Z39.88-2004&rft_val_fmt=info:ofi/fmt:kev:mtx:dissertation&res_dat=xri:pqm&rft_dat=xri:pqdiss:3597558},
}

@article{OkBethe,
author = {Aganagic, Mina and Okounkov, Andrei},
year = {2016},
month = {01},
pages = {565-600},
title = {{Cuasimap Counts and Bethe Eigenfunctions}},
volume = {16},
journal = {Moscow Mathematical Journal},
doi = {10.17323/1609-4514-2016-16-4-565-600}
}

@article {NegGor,
    AUTHOR = {Gorsky, Eugene and Negu\c t, Andrei},
     title = {{Infinitesimal change of stable basis}},
   JOURNAL = {Selecta Math. (N.S.)},
  FJOURNAL = {Selecta Mathematica. New Series},
    VOLUME = {23},
      YEAR = {2017},
    NUMBER = {3},
     PAGES = {1909--1930},
      ISSN = {1022-1824},
   MRCLASS = {17B37 (14C05 14C35 19G99)},
  MRNUMBER = {3663597},
       DOI = {10.1007/s00029-017-0327-5},
       URL = {http://dx.doi.org/10.1007/s00029-017-0327-5},
}

@article{AOElliptic,
      author         = "Aganagic, Mina and Okounkov, Andrei",
      title = {{Elliptic stable envelopes}},
      year           = "2016",
      eprint         = "1604.00423",
      archivePrefix  = "arXiv",
      primaryClass   = "math.AG",
      SLACcitation   = "%%CITATION = ARXIV:1604.00423;%%"
}

@incollection {GinzburgLectures,
    AUTHOR = {Ginzburg, Victor},
     title = {{Lectures on Nakajima's quiver varieties}},
 BOOKtitle = {{Geometric methods in representation theory. I}},
    SERIES = {S\'emin. Congr.},
    VOLUME = {24},
     PAGES = {145--219},
 PUBLISHER = {Soc. Math. France, Paris},
      YEAR = {2012},
   MRCLASS = {14L24 (16G20 17B67)},
  MRNUMBER = {3202703},
MRREVIEWER = {Xueqing Chen},
}

@incollection {NakajimaLectures2,
    AUTHOR = {Nakajima, Hiraku},
     title = {{More lectures on Hilbert schemes of points on surfaces}},
 BOOKtitle = {{Development of moduli theory---Kyoto 2013}},
    SERIES = {Adv. Stud. Pure Math.},
    VOLUME = {69},
     PAGES = {173--205},
 PUBLISHER = {Math. Soc. Japan, [Tokyo]},
      YEAR = {2016},
   MRCLASS = {14C05 (14D21 14J60)},
  MRNUMBER = {3586508},
}

@book {NakajimaLectures1,
    AUTHOR = {Nakajima, Hiraku},
     title = {{Lectures on Hilbert schemes of points on surfaces}},
    SERIES = {University Lecture Series},
    VOLUME = {18},
 PUBLISHER = {American Mathematical Society, Providence, RI},
      YEAR = {1999},
     PAGES = {xii+132},
      ISBN = {0-8218-1956-9},
   MRCLASS = {14C05 (17B69)},
  MRNUMBER = {1711344},
MRREVIEWER = {Mark Andrea A. de Cataldo},
       DOI = {10.1090/ulect/018},
       URL = {http://dx.doi.org/10.1090/ulect/018},
}

@article{SmirnovElliptic,
author = {Smirnov, Andrey},
year = {2018},
month = {04},
pages = {},
title = {{Elliptic stable envelope for Hilbert scheme of points in the plane}},
volume = {26},
journal = {Selecta Mathematica},
doi = {10.1007/s00029-019-0527-2}
}

@book {mumf,
	AUTHOR = {Mumford, David},
	title = {{Tata lectures on theta. I}},
	SERIES = {Modern Birkh\"auser Classics},
	NOTE = {With the collaboration of C. Musili, M. Nori, E. Previato and
	M. Stillman,
	Reprint of the 1983 edition},
	PUBLISHER = {Birkh\"auser Boston, Inc.},
	YEAR = {2007},
	PAGES = {xiv+235},
	ISBN = {978-0-8176-4572-4; 0-8176-4572-1},
	MRCLASS = {14K25 (11E45 11G10 14C30)},
	MRNUMBER = {2352717},
	URL = {https://doi.org/10.1007/978-0-8176-4578-6},
}

@ARTICLE{KononovSmirnov1,
	author = {{Kononov}, Yakov and {Smirnov}, Andrey},
	title = "{Pursuing quantum difference equations I: stable envelopes of subvarieties}",
	journal = {arXiv e-prints},
	year = 2020,
	month = apr,
	eid = {arXiv:2004.07862},
	pages = {arXiv:2004.07862},
	archivePrefix = {arXiv},
	eprint = {2004.07862},
	primaryClass = {math.RT},
	adsurl = {https://ui.adsabs.harvard.edu/abs/2020arXiv200407862K}
}

@misc{smirnov2021quantum,
      title={Quantum differential and difference equations for Hilb},
      author={Andrey Smirnov},
      year={2021},
      eprint={2102.10726},
      archivePrefix={arXiv},
      primaryClass={math.AG}
}

@article {NakajimaBows,
	AUTHOR = {Nakajima, Hiraku and Takayama, Yuuya},
	TITLE = {Cherkis bow varieties and {C}oulomb branches of quiver gauge
	theories of affine type {$A$}},
	JOURNAL = {Selecta Math. (N.S.)},
	FJOURNAL = {Selecta Mathematica. New Series},
	VOLUME = {23},
	YEAR = {2017},
	NUMBER = {4},
	PAGES = {2553--2633},
	ISSN = {1022-1824},
	MRCLASS = {16G20 (14D21 81T13)},
	MRNUMBER = {3703461},
	MRREVIEWER = {Tristan Bozec},
	DOI = {10.1007/s00029-017-0341-7},
	URL = {https://doi.org/10.1007/s00029-017-0341-7},
}

@article {OkTalk,
	AUTHOR = {{Aganagic}, Mina and  {Okounkov}, Andrei},
	TITLE = {Duality interfaces in 3-dimensional theories, talks at \textit{String Math 2019}, 
	available at \textsf{https://www.stringmath2019.se/scientific-talks-2/. pages 7, 17}}
}

@article {OkRodeIsland,
	AUTHOR = {{Okounkov}, Andrei},
	TITLE = {Informal talks at the AMS Mathematics Research Community meeting on Geometric Representation Theory and Equivariant elliptic cohomology, \textit{Rode Island in June 2019}}
}

@article{AOprep,
author = {Aganagic, Mina and Okounkov, Andrei},

title = {{In preparation.}},

}
	\endgroup
\end{tiny}

\addcontentsline{toc}{chapter}{References}  


\titleformat{\chapter}[display]
{\normalfont\bfseries\filcenter}{}{0pt}{\large\chaptertitlename\ \large\thechapter : \large\bfseries\filcenter{#1}}  
\titlespacing*{\chapter}
  {0pt}{0pt}{30pt}	
  
\titleformat{\section}{\normalfont\bfseries}{\thesection}{1em}{#1}

\titleformat{\subsection}{\normalfont}{\thesubsection}{0em}{\hspace{1em}#1}

\begin{appendices}

\addtocontents{toc}{\protect\renewcommand{\protect\cftchappresnum}{\appendixname\space}}
\addtocontents{toc}{\protect\renewcommand{\protect\cftchapnumwidth}{6em}}


\chapter{Wiener-Hopf factorization of matrices with rational entries}

If a matrix has several singular points, we would like to factorize it to a product of matrices each having a singularity at a single point. Let us start with the simplest case of two singularities at $0$ and $\infty$.
\begin{lemma}
Suppose
\[
M(z) \in \End(\C^n) \otimes \C[z^\pm 1]
\]
is a matrix, such that $\det M(z) \neq 0$ for all $z \in \C^\times$. 
\end{lemma}
Then there exist matrices 
\[
A(z) \in \End(\C^n) \otimes \C[z], \ \ B(z) \in \End(c^n) \otimes \C[z^{-1}]
\]
and a diagonal matrix
\[
D(z) = \begin{pmatrix}
z^{d_1} & & \\
 & \ddots & \\
 & & z^{d_n}
\end{pmatrix}
\]
with the following properties:
\begin{enumerate}
    \item $A(z)$ is nondegenerate:
    \[
    \det(A(z)) \neq 0,
    \]
    including the point $z=0$,
    \item $B(z)$ is nondegenerate:
    \[
    \det B(z) \neq 0,
    \]
    including the point $z = \infty$,
    \item
    They provide a factorization
    \[
    M(z) = A(z) \cdot D(z) \cdot B(z).
    \]
\end{enumerate}
\begin{proof}
The proof is beautiful, nonconstructive, and uses simple algebraic geometry.
Consider the projective line $\P^1$ over $\C$ and its \v{C}ech covering by affine charts
\[
U_1 = \P^1 \setminus \{\infty\}, \ \ U_2 = \P^1 \setminus \{0\}, \ \ U_1 \cap U_2 = \C^\times.
\]
Vector bundles over $\P^1$ are can be defined by a gluing function, in particular, the function
\[
M(z): U_1 \cap U_2 \to GL(n, \cO)
\]
defines a vector bundle $\V$ over $\P^1$.

By the theorem of Grothendieck, any vector bundle over $\P^1$ is isomorphic to a direct sum of line bundles, thus
\[
\V \cong \bigoplus_{i=1}^n \cO(d_i).
\]
The matrices $A(z)$ and $B(z)$ are the matrices of change of trivialization in $U_1$ and $U_2$, and the diagonal matrix of monomials in $z$ is the gluing function for the bundle  $\bigoplus_{i=1}^n \cO(d_i)$.
\end{proof}

Note that the diagonal matrix is uniquely determined up to conjugation, and its determinant is equal related to the characteristic class
\[
\det D = z^{c_1(\det \V)}.
\]
The matrices $A(z)$ and $B(z)$ are determined up to conjugation by a constant matrix .
Indeed, suppose we have two factorizations
\[
A(z) D(z) B(z) = A_1(z) D_1(z) B_1(z),
\]
then
\[
D_1(z)^{-1} A_1(z)^{-1} A(z) D(z) = B_1(z) B(z)^{-1}.
\]
The right hand side is a matrix with coefficients in $\C[z^{-1}]$ which is nondegenerate for all $z \neq 0$. But the left hand side is nondegenerate at $z=0$, that is why the matrix $B_1(z) B_1^{-1}(z)$ must be constant.

For a more general rational matrix $M(z)$ let us say that it is degenerate at $z=z_0$ if the value $M(z_0)$ is not finite or
\[
\det M(z_0) = 0.
\]

Now we will use complex-analytic analogue of the same idea.
\begin{lemma}
Suppose the points where M(z) is degenerate are $z_1, ... , z_k$ and $0, \infty$,
and 
\[
|z_1| < ... < |z_k|.
\]
Then there is a factorization
\[
M(z) = A(z) \cdot D(z) \cdot B(z)
\]
such that
\begin{enumerate}
    \item $A(z)$ has singularities only in $\{z_k, \infty\}$,
    \item $D(z)$ is a diagonal matrix,
    \[
D(z) = \begin{pmatrix}
z^{d_1} & & \\
 & \ddots & \\
 & & z^{d_n}
\end{pmatrix}
\]
\item $B(z)$ has singularities only in $\{0, z_1, ..., z_{k-1}\}$
\end{enumerate}
\end{lemma}
\begin{proof}
Choose $R \in \R$ such that for all $z_i$:
\[
|z_i|<R<|z_k|.
\]
Consider \v{C}ech covering of $\P^1$ by complex-analytic charts
\[
U_1 = \P^1 \setminus \{\infty\}, \ \ U_2 = \{z: |z|>R\}.
\]

As in the proof of the previous lemma, $M(z)$ defines the transition matrix of a vector bundle. By the same argument (which works in the complex-analytic setting), we obtain a factorization
\[
M(z) = A(z) D(z) B(z),
\]
where $A(z)$ and $B(z)$ are analytic nondegenerate functions on $U_1$ and $U_2$ respectively, and $D(z)$ is a diagonal matrix of monomials in $z$.

We need to prove that $A(z)$ and $B(z)$ are rational matrices. For $A(z)$ it is enough to check that its matrix elements grow not faster than polynomially when $z\to \infty$ and $z\to z_k$. In the neighborhood of infinity, $A(z)$ can be approximated as
\[
A(z) \sim M(z) B(\infty)^{-1} D(z)^{-1},
\]
which proves the claim. For $B(z)$ we need to check that it is meromorphic in neighborhoods of $z_i$ and $0$. It follows from the same argument:
\[
B(z) \sim A(z_i)^{-1} D(z)^{-1} M(z), \ \ z\to z_i.
\]
\end{proof}

The lemma allows us to separate a singularity at $\infty$ from other singularities.
Applying this lemma by induction, we get the following factorization theorem
\begin{proposition}
Suppose $M(z)$ is a rational matrix with entries in $\C^\times$ with singularities at $z_1, ..., z_n$ such that
\[
|z_1| < ... < |z_n|.
\]
Then $M(z)$ can be factorized as
\[
M(z) = D(z) A_1(z) A_2(z) ... A_n(z),
\]
where $A_i(z)$ is a matrix whose singularities are contained in $\{z_i, 0, \infty\}$, and $D(z)$ having singularities in $\{0, \infty\}$
\end{proposition}

\end{appendices}

\end{document}